\documentclass{siamart250211}

\usepackage{microtype}
\usepackage{amssymb}
\usepackage{bm}
\usepackage{multirow} 
\usepackage{amsmath}

\usepackage{epstopdf}
\ifpdf
\DeclareGraphicsExtensions{.eps,.pdf,.png,.jpg}
\else
\DeclareGraphicsExtensions{.eps}
\fi
\usepackage{xcolor}
\usepackage{tikz}
\usepackage{float}
\usepackage{graphicx}
\usepackage{booktabs}

\usepackage{caption}        
\usepackage{subcaption}

 \usepackage{siunitx}    
 \usepackage{array}       

\usepackage{algorithmic}
\usepackage{tabularx}
\newsiamremark{remark}{Remark}
\newsiamthm{coro}{Corollary}
\newsiamthm{prop}{Proposition}

\newcommand{\newsiamdefn}[2]{
	\theoremstyle{plain}
	\theoremheaderfont{\normalfont\sc}
	\theorembodyfont{\normalfont}
	\theoremseparator{.}
	\theoremsymbol{}
	\newtheorem{#1}{#2}
}
\newsiamdefn{defn}{Definition}
\newsiamdefn{amp}{Assumption}
\newsiamdefn{example}{Example}

\usepackage{amsfonts}
\usepackage{mathtools}
\usepackage{nicefrac}
\usepackage{mathrsfs}
\usepackage{nicematrix}
\usepackage{amsopn}

\allowdisplaybreaks[4]

\headers{Quasi-orthogonal method for Schr{\"o}dinger eigenpairs}{S. Wang, and A. Zhou}

\title{A quasi-orthogonal method based on the inverse operator
		\\for computing many Schr{\"o}dinger eigenpairs\thanks{
		\funding{This work was supported by the National Natural Science Foundation of China under
				grant 12571446 and the National Key R \& D Program of China under grants 2025YFA1016600 and 2025YFA1016601.}}
}

\author{
	Shengyue Wang\thanks{SKLMS, Academy of Mathematics and Systems Science, Chinese Academy of Sciences, Beijing 100190, China; and School of Mathematical Sciences, University of Chinese Academy of Sciences, Beijing 100049, China (\email{wangshengyue@amss.ac.cn}, \email{azhou@lsec.cc.ac.cn}).}
	\and Aihui Zhou\footnotemark[2]
}

\begin{document}
	
	\maketitle

	% Abstract

\begin{abstract}
The simultaneous computation of many Schr{\"o}dinger eigenpairs requires mutually orthogonal eigenfunctions. We propose a quasi-orthogonal evolution model based on the inverse operator and a two-stage time discretization for nonorthogonal initial data. The continuous trajectory and discrete iterates remain full rank and approach orthogonality without any orthogonalization. The discretization preserves the current Gram matrix in the first stage and reduces the orthogonality error in the second. We establish well-posedness, energy decay, asymptotic orthogonality, and exponential convergence to the target eigenfunctions up to fixed orthogonal transformations. Numerical experiments for two Schr{\"o}dinger eigenvalue problems validate the decay of the energy, gradient norm, and relative iterate error, while the iterates become numerically orthogonal. Further experiments show that the discrete time steps remain effective under mesh refinement, that convergence rates depend on the spectral gap, and that dynamically adjusting the solve accuracy can accelerate the computation.
\end{abstract}

	% Keywords
	
	% MSC codes
	
	\section{Introduction}\label{sec: intro}
		The simultaneous computation of many Schr{\"o}dinger eigenpairs is
		an important task in computational quantum mechanics and condensed
		matter physics \cite{lehtovaara2007solving,saad2011numerical,trefethen1997numerical}.
		The corresponding eigenfunctions must be mutually orthogonal.
		In Kohn--Sham density functional theory, for example, the occupied
		orbitals entering the electron density satisfy orthogonality
		constraints associated with the fermionic structure of the model
		\cite{dai2020,dai2021convergent,martin2020electronic,zhang2014gradient}.
		Spatial discretization typically leads to large algebraic eigenvalue
		problems \cite{babuska1991eigenvalue,cances2023density,chen2014adaptive,saad2010numerical}.
		
	For algebraic eigenvalue problems, standard Krylov subspace methods
	and block eigensolvers construct a search
	subspace and extract eigenpair approximations through a projected eigenvalue
	problem \cite{knyazev2001toward,sorensen1992implicit}.
	For Lanczos-type computations, loss of orthogonality can lead to redundant
	copies of converged eigenvectors \cite{parlett1979lanczos}.
	Gram--Schmidt or
	QR-type procedures are commonly used to maintain the linear independence and
	mutual orthogonality of the basis vectors. These procedures require repeated
	inner products whose cost grows with the dimension of the search subspace
	\cite{bjorck1994numerics,smoktunowicz2006note,walker1988implementation}.
	In parallel implementations, the associated reductions also introduce
	synchronization \cite{hernandez2005slepc,hwang2010parallel}. This motivates
	iterations that obtain mutually orthogonal eigenfunction approximations
	without additional orthogonalization.

Evolution equations offer an alternative framework for formulating and solving eigenvalue problems. In many settings, the desired states or eigenfunctions are characterized by minimizing an energy functional under orthogonality constraints, which leads naturally to constrained gradient flows. Dai et al. developed an extended gradient flow for Kohn--Sham energy minimization and a class of orthogonality-preserving discretizations, including a midpoint scheme. They established energy decay and convergence under suitable step-size conditions \cite{dai2020,dai2021convergent}. Normalized and Sobolev gradient flows have also been studied for the Gross--Pitaevskii eigenvalue problem \cite{bao2004computing,chen2024convergence,henning2020sobolev}. Energy minimization for several states has been studied on the Stiefel manifold \cite{altmann2022energy}. For linear eigenvalue problems, Chu et al. incorporated the inverse operator into an orthogonality-preserving evolution model and time discretization. Their formulation preserves orthogonality, dissipates the energy, and admits time-step bounds independent of the spatial discretization \cite{chu2025orthogonality}. The orthogonality-preserving discretizations of Dai et al. and Chu et al. require no additional orthogonalization when initialized with orthogonal data. These discretizations preserve the existing Gram matrix rather than drive it toward the identity.

Quasi-orthogonal evolution equations and their iterative schemes
were introduced in \cite{wang2025quasi,wang2026quasi} to approach
orthogonality from nonorthogonal initial data without any
orthogonalization.
However, the step-size conditions for the corresponding temporal discretizations
depend on the spatial discretization.
In this paper, we propose and analyze a quasi-orthogonal method
based on the inverse operator to address this dependence.
We construct a continuous evolution equation and a two-stage time
discretization that retain the quasi-orthogonal structure:
starting from nonorthogonal initial data in a quasi-Stiefel set,
the continuous trajectory and discrete iterates remain full rank,
and their orthogonality errors tend to zero.
The first stage preserves the current Gram matrix, while the second
reduces the orthogonality error.
Numerical experiments show that the same time steps remain effective
under mesh refinement.
Figures \ref{fig:utgeqin} and \ref{fig:ugeqin} illustrate the
continuous trajectory and the discrete iterates, respectively.

\begin{figure}[ht]
	\centering
	\begin{minipage}{0.48\textwidth}
		\centering
		\includegraphics[width=0.95\linewidth]{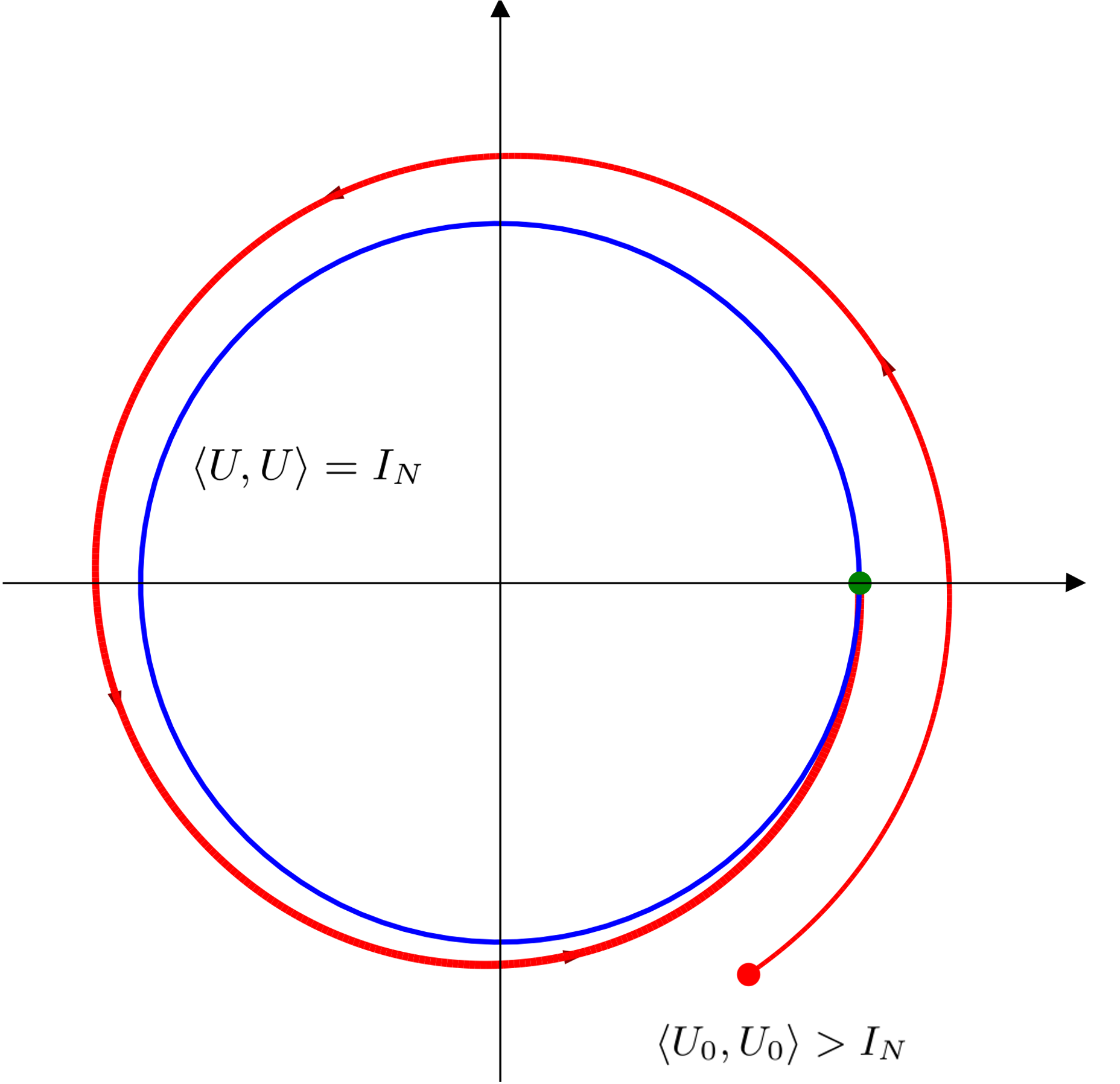}
		\caption{Quasi-orthogonal trajectory of \eqref{equ: evolution problem}}
		\label{fig:utgeqin} 
	\end{minipage}
	\hfill
	\begin{minipage}{0.48\textwidth}
		\centering
		\includegraphics[width=0.98\linewidth]{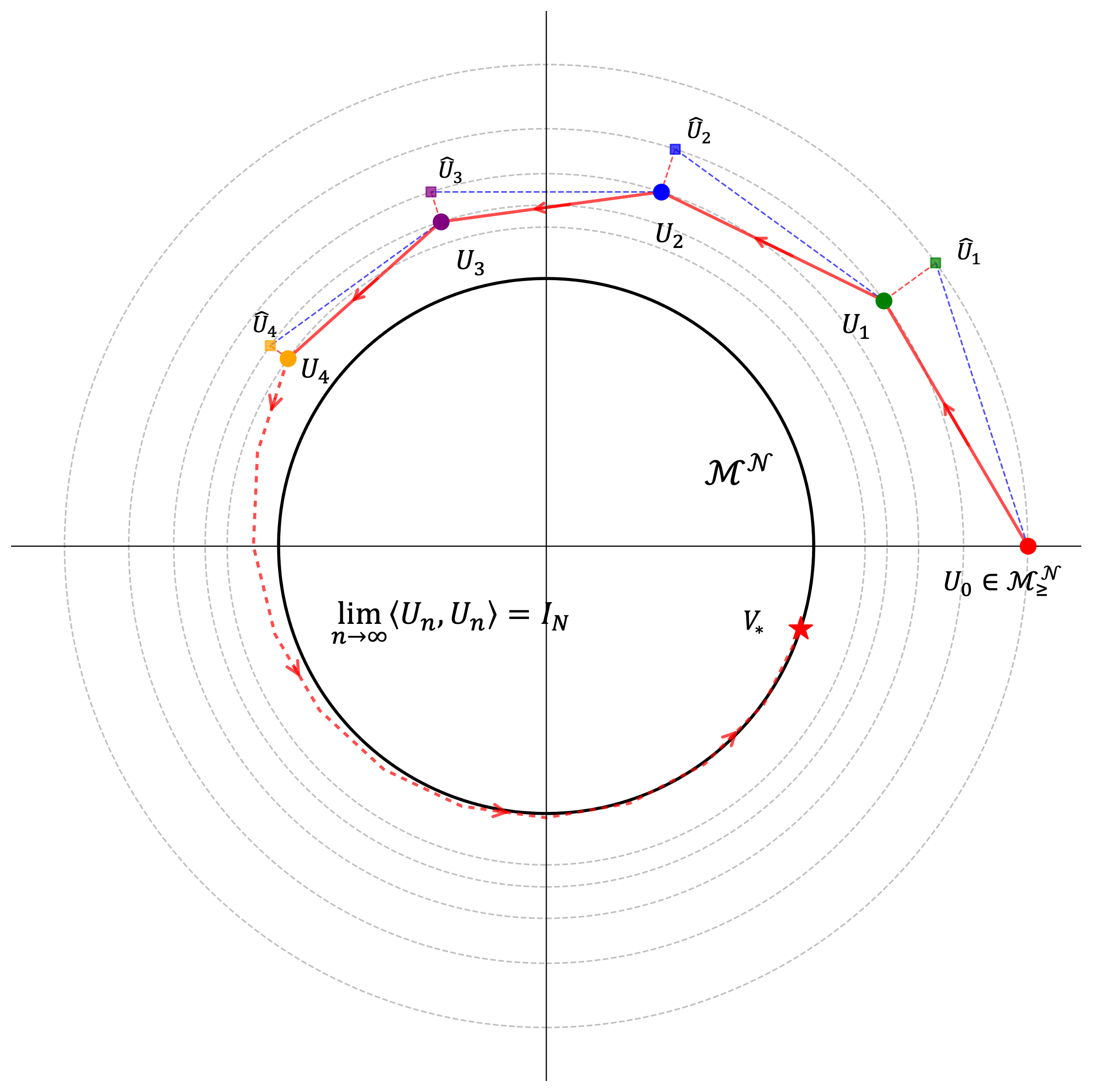}
		\caption{Quasi-orthogonal iterates generated by \eqref{equ: numerical scheme}}
		\label{fig:ugeqin} 
	\end{minipage}
\end{figure}

The main contributions in this paper are as follows.
\begin{itemize}
		\item For the inverse-operator evolution equation, we prove global well-posedness,
		preservation of the quasi-Stiefel set, energy decay, and asymptotic
		orthogonality. We further establish exponential convergence of the
		gradient, trajectory, and energy.
		\item For the two-stage discretization, we prove well-posedness,
		preservation of the quasi-Stiefel set, and energy decay under suitable
		time-step conditions. We further establish
		asymptotic orthogonality and exponential convergence of the discrete
		iterates.
		\item Numerical experiments for two
		model problems exhibit energy decay, convergence of the iterates, and
		numerical orthogonality. Further experiments confirm mesh-independent iteration
		counts for the time steps, dependence of convergence rates on the
		inverse spectral gap, and
		a reduction in the source-solve cost through dynamic tolerances.
\end{itemize}

The remainder of the paper is organized as follows.
Section \ref{sec: preliminaries} introduces the notation and
formulates the eigenvalue problem.
Sections \ref{sec: continuous model} and
\ref{section:Time discretization} present the continuous
quasi-orthogonal model and its time discretization, respectively,
and analyze their properties.
Section \ref{section:Numerical experiments} presents the numerical
experiments.
Section \ref{sec: conclusion} concludes the paper.
The proofs of the auxiliary lemmas are collected in
Appendix \ref{app:lemma-proofs}.

	\section{Preliminaries}\label{sec: preliminaries}

	\subsection{Notation}

	We first introduce the function spaces, inner products, and matrix notation
	for collections of functions.
	
	Let $\Omega\subset\mathbb{R}^{d}$ $(d\in\mathbb{N}_{+})$ be a bounded domain with boundary $\partial\Omega$. We denote by $H^1(\Omega)$ the standard Sobolev space and define
	$$H_0^1(\Omega) = \{ v \in H^1(\Omega) : v = 0 \text{ on } \partial\Omega \}.$$
	The $L^{2}(\Omega)$ inner product and its induced norm are denoted by $(\cdot,\cdot)$ and $\|\cdot\|$, respectively, where
	
	$$(u,v) = \int_{\Omega} uv \, dx, \qquad \|u\| = \sqrt{(u,u)}.$$
	
	Let $\mathcal{V}:\Omega\rightarrow\mathbb{R}$ be a potential. The symmetric
	bilinear form associated with the Schr{\"o}dinger operator
	$-\Delta+\mathcal V$ is
	$a(\cdot,\cdot):H_{0}^{1}(\Omega)\times H_{0}^{1}(\Omega)\rightarrow\mathbb{R}$,
	defined by
	$$a(u,v) \coloneqq \int_{\Omega} \left( \nabla u \cdot \nabla v + \mathcal{V}uv \right) dx.$$
	We assume that $a(\cdot,\cdot)$ is symmetric, bounded, and coercive. Thus,
	there exist constants $c_1,c_2>0$ such that
	$$c_1 \|\nabla u\|^2 \leqslant a(u,u) \leqslant c_2 \|\nabla u\|^2, \qquad \forall\, u \in H_0^1(\Omega).$$
	The form $a(\cdot,\cdot)$ is then an inner product on
	$H_0^1(\Omega)$, with norm
	$\|u\|_a\coloneqq\sqrt{a(u,u)}$. The Poincar\'e inequality gives a
	constant $c_\Omega>0$ such that
	$$\|u\| \leqslant c_{\Omega}\|u\|_a, \quad \forall\, u \in H_0^1(\Omega).$$
	
	For $N\in\mathbb{N}_{+}$, define the product spaces
	$(L^{2}(\Omega))^{N}$ and $(H_{0}^{1}(\Omega))^{N}$ by
	$$\begin{aligned}
		& \big(L^2(\Omega)\big)^N = \left\{\left(u_1, u_2, \cdots, u_N\right): u_i\in L^2(\Omega) , i = 1, 2, \cdots,N\right\}, \\
		& \big(H_0^1(\Omega)\big)^N = \left\{\left(u_1, u_2, \cdots, u_N\right): u_i\in H_0^1(\Omega) , i = 1, 2, \cdots,N\right\}.
	\end{aligned}$$
	
	For any $U,V\in(L^{2}(\Omega))^{N}$, we define the $L^{2}$ inner product matrix $\langle U,V\rangle\in\mathbb{R}^{N\times N}$ and the associated global inner product $(U,V)$ as
	$$\langle U, V \rangle \coloneqq \Big((u_i, v_j)\Big)_{i, j=1}^N \quad \text{and} \quad (U, V) \coloneqq \operatorname{tr}(\langle U, V \rangle) = \sum_{i=1}^N (u_i, v_i),$$
	with the corresponding norm $\|U\|\coloneqq\sqrt{(U,U)}$. In particular, $\langle U,U\rangle$ is the Gram matrix of $U$. Similarly, for $U,V\in(H_{0}^{1}(\Omega))^{N}$, the inner product matrix $\langle U,V\rangle_{a}\in\mathbb{R}^{N\times N}$ and the inner product $(U,V)_{a}$ are given by
	$$\langle U, V \rangle_a \coloneqq \Big(a(u_i, v_j)\Big)_{i, j=1}^N \quad \text{and} \quad (U, V)_a \coloneqq \operatorname{tr}(\langle U, V \rangle_a),$$
	with the induced norm $\|U\|_{a}\coloneqq\sqrt{(U,U)_{a}}$.
	
	The Stiefel manifold associated with the $L^2$ inner product is
	$$\mathcal{M}^N \coloneqq \left\{ U \in \big(H_0^1(\Omega)\big)^N : \langle U, U \rangle = I_N \right\},$$
	where $I_N$ is the $N\times N$ identity matrix. The Grassmann manifold is
	$\mathcal G^N=\mathcal M^N/\mathord\sim$, where
	$U\sim\widehat U$ if and only if $\widehat U=UQ$ for some
	$Q\in\mathcal O^N$, where $\mathcal O^N$ denotes the set of
	$N\times N$ orthogonal matrices. For $U\in(H_0^1(\Omega))^N$, write
	$$[U] \coloneqq \{ UQ : Q \in \mathcal{O}^N \},$$
	and thus $\mathcal{G}^{N}=\{[U]:U\in\mathcal{M}^{N}\}$.
	
	The distance between two equivalence classes is
	$$\| [U] - [V] \|_a \coloneqq \min_{Q \in \mathcal{O}^N} \| U - VQ \|_a.$$
	
	For symmetric matrices $A,B\in\mathbb R^{N\times N}$, we write
	$A\leqslant B$ if $\xi^\top A\xi\leqslant \xi^\top B\xi$ for every
	$\xi\in\mathbb R^N$. We denote the spectrum of $A$ by $\lambda(A)$ and
	its smallest and largest eigenvalues by $\lambda_{\min}(A)$ and
	$\lambda_{\max}(A)$, respectively.
	
	\subsection{Problem setting}

	With this notation, we formulate the eigenvalue problem and its
	associated constrained energy formulation.
	
	We seek the $N$ smallest eigenvalues and corresponding mutually orthogonal
	eigenfunctions. Find $(u_i,\lambda_i)\in H_0^1(\Omega)\times\mathbb R$,
	$i=1,\ldots,N$, such that
	$$\left\{\begin{aligned}
		&a(u_i, v) = \lambda_i (u_i, v), \quad \forall\, v \in H_0^1(\Omega), \\
		&(u_i, u_j) = \delta_{ij}.
	\end{aligned}\right.$$
	With $U=(u_1,\ldots,u_N)$ and
	$\Lambda=\operatorname{diag}(\lambda_1,\ldots,\lambda_N)$, this system is
	equivalent to
	\begin{equation}\label{equ: eigenvalue problem}
	\left\{\begin{aligned}
	&	\langle U, V \rangle_a = \langle U \Lambda, V \rangle, \quad \forall\, V \in \big(H_0^1(\Omega)\big)^N,
		\\& \langle U, U \rangle = I_N.
	\end{aligned}\right.
	\end{equation}
	We assume that the eigenvalues are ordered as $\lambda_{1}\leqslant\lambda_{2}\leqslant\dots\leqslant\lambda_{N}\  (<\lambda_{N+1})$, and we denote by $V_{*}=(v_{1},\dots,v_{N})$ a solution to \eqref{equ: eigenvalue problem} corresponding to the first $N$ eigenvalues.
	
	We define the energy associated with the Schr{\"o}dinger operator by
	\begin{equation*}
		E(U) \coloneqq \frac{1}{2}(U,U)_a.
	\end{equation*}
	The eigenfunctions $V_*$ solve the constrained minimization problem
	\begin{equation}\label{equ: minimization problem in Stiefel}
		E(V_*) = \min_{U \in \mathcal{M}^N} E(U).
	\end{equation}
	Since $E(UQ)=E(U)$ for every $Q\in\mathcal O^N$, problem
	\eqref{equ: minimization problem in Stiefel} is equivalent to
	\begin{equation}
	\label{equ: minimization problem in Grassmann} 
	E([V_*]) = \min_{[U] \in \mathcal{G}^N} E(U).
	\end{equation}
	The spectral gap $\lambda_N<\lambda_{N+1}$ makes the invariant subspace
	spanned by the first $N$ eigenfunctions unique. Hence $[V_*]$ is the
	unique minimizer of \eqref{equ: minimization problem in Grassmann}.

We next reformulate the eigenvalue problem using the inverse operator.
Let $\mathcal G\coloneqq(-\Delta+\mathcal V)^{-1}:H^{-1}(\Omega)
	\to H_0^1(\Omega)$.
For $U,V\in\big(H_0^1(\Omega)\big)^N$, the inverse operator satisfies
	\begin{align*}
		(\mathcal{G} U, V)_a = (\mathcal{G} V, U)_a = (U, V).
	\end{align*}

		The inverse operator has the same eigenfunctions as the
		Schr{\"o}dinger operator, with reciprocal eigenvalues.
		Thus, the smallest eigenvalues of the original problem correspond
		to the largest eigenvalues of $\mathcal G$.
		We therefore consider the inverse eigenvalue problem:
		\begin{equation}\label{eq: inverse eigenvalue problem}
			\left\{ \begin{aligned}
				&\langle \mathcal{G}U, V\rangle = \langle U\Lambda^{-1}, V\rangle, \quad \forall\, V \in \big(H_0^1(\Omega)\big)^N, \\
				&U \in \mathcal{M}^N.
			\end{aligned}\right.
		\end{equation}
		The corresponding maximization problem formulated on the Grassmann manifold reads
		\begin{equation}\label{eq: inverse maximum problem}
			\max_{[U] \in \mathcal{G}^N} E_{\mathcal{G}}(U) \coloneqq \frac{1}{2} (U, \mathcal{G}U).
		\end{equation}
		$[V_*]$ is the unique maximizer of
		\eqref{eq: inverse maximum problem}.
		By the Grassmann geometry in \cite{edelman1998geometry}, the gradient of
		$E_{\mathcal G}$ at $[U]\in\mathcal G^N$ is
		\begin{equation*}
			\nabla_G E_{\mathcal{G}}(U) = \mathcal{G} U - U \langle \mathcal{G} U, U \rangle.
		\end{equation*}
		For nonorthogonal $U\in\big(H_0^1(\Omega)\big)^N$, we use the same
		expression to define $\nabla_GE_{\mathcal G}(U)$.
        In particular,
		$\nabla_GE_{\mathcal G}(V_*)=0$, so $V_*$ is a steady state.

%	The analysis also applies to bilinear forms satisfying a G\aa rding
%	inequality (see \cite[Remark 2.9]{dai2008convergence}). In this case, we
%	assume that there exists $C>0$ such that
%	$$a(u, u) \geqslant \frac{1}{2} \|\nabla u\|^2 - C \|u\|^2, \qquad \forall\, u \in H_0^1(\Omega).$$
%	
%	
	
	\section{A quasi-orthogonal model}\label{sec: continuous model}
	
	Using the variational formulation of Section \ref{sec: preliminaries}, we
	introduce the continuous quasi-orthogonal model. We then establish its
	well-posedness and analyze its long-time behavior.
	
	\subsection{The evolution equation}

	The inverse-operator maximization problem
	\eqref{eq: inverse maximum problem} motivates the following
	evolution equation, allowing nonorthogonal initial data:
	\begin{equation}\label{equ: evolution problem}
		\left\{  \begin{aligned}
			&\frac{\mathrm{d}U}{\mathrm{d}t} = \mathcal{G} U - U \langle \mathcal{G} U, U \rangle \coloneqq \nabla_G E_{\mathcal{G}}(U), 
			\\& U(0) = U_0\in \big(H_0^1(\Omega)\big)^N.
		\end{aligned}\right.
	\end{equation}
	
	The next representation follows from \cite[Theorem 3.10]{wang2025quasi}.
	\begin{theorem}\label{expression of solution}
		Suppose that the solution $U(t)$ of \eqref{equ: evolution problem}
		exists on $[0,T)$ for some $T>0$. Then
		\begin{equation}\label{eq:expression of solution}
			U(t) = \exp(\mathcal{G}t) U_0 \left[ I_N - \langle U_0, U_0 \rangle + \langle U_0, \exp(2\mathcal{G}t) U_0 \rangle \right]^{-1/2} Q(t),
		\end{equation}
		where $Q(t) \in \mathcal{O}^N$ is an orthogonal matrix depending on $t$.
	\end{theorem}
	
	Theorem \ref{expression of solution} shows that
	$\langle U(t),U(t)\rangle=I_N$ whenever
	$\langle U_0,U_0\rangle=I_N$. For nonorthogonal initial data, the
	following quasi-orthogonality result is proved in
	\cite[Lemma 4.1]{wang2025quasi}.
	
	\begin{theorem}\label{thm:U in quasi-Stiefel}
		Suppose that the solution $U(t)$ to \eqref{equ: evolution problem} exists on $[0, T)$ for some $T > 0$. If $U_0$ satisfies 
		$$U_0\in  \mathcal{M}_{\geqslant}^N \coloneqq \left\{U\in \big(H_0^1(\Omega)\big)^N: \left\langle U, U \right\rangle \geqslant I_N\right\},$$
		where $\mathcal{M}_{\geqslant}^N$ is referred to as the quasi-Stiefel set, then
		\begin{equation*}
			\begin{aligned}
				U(t) \in \mathcal{M}_{\geqslant}^{N}, \quad \forall\, t \in [0, T).
		\end{aligned}
		\end{equation*}
	\end{theorem}
	
	The flow also dissipates the energy.
	
	\begin{theorem}\label{thm:energy decay}
		Let $U(t)$ be a solution to \eqref{equ: evolution problem} on $[0, T)$ with initial data $U_0 \in \mathcal{M}_{\geqslant}^{N}$. Then the energy is non-increasing over time, that is,
		\begin{equation*}
			\frac{\mathrm{d}}{\mathrm{d}t} E(U(t)) \leqslant 0.
		\end{equation*}
	\end{theorem}
	\begin{proof}
		Taking the time derivative of the energy yields
		\begin{equation*}
			\begin{aligned}
				\frac{\mathrm{d}   E(U(t)) }{\mathrm{d}  t}  = & \left( U(t), \frac{\mathrm{d}  U(t) }{\mathrm{d}  t} \right)_a
				\\ = & \left( U(t) - \mathcal{G} U(t) \langle \mathcal{G} U(t),  U(t) \rangle^{-1}, \mathcal{G} U(t)- U(t) \langle \mathcal{G} U(t),  U(t) \rangle \right)_a
				\\& + \left(  \mathcal{G} U(t) \langle \mathcal{G} U(t),  U(t) \rangle^{-1}, \mathcal{G} U(t)- U(t) \langle \mathcal{G} U(t),  U(t) \rangle \right)_a
				\\ = & - \left(\nabla_G E_{\mathcal{G}}(U(t))  \langle \mathcal{G} U(t),  U(t) \rangle^{-1}, \nabla_G E_{\mathcal{G}}(U(t)) \right)_a
				\\& +\underbrace{\left(U(t) \langle \mathcal{G} U(t),  U(t) \rangle^{-1}, \mathcal{G}U(t) -U(t) \langle \mathcal{G} U(t),  U(t) \rangle \right)}_{\mathrm{I}},
			\end{aligned}
		\end{equation*}
	where
	\begin{equation*}
		\begin{aligned}
			\mathrm{I} &= \left( U(t) \langle \mathcal{G} U(t), U(t) \rangle^{-1}, \mathcal{G} U(t) \right) - \left( U(t) \langle \mathcal{G} U(t), U(t) \rangle^{-1}, U(t) \langle \mathcal{G} U(t), U(t) \rangle \right)
				\\&= -  \operatorname{tr}\left(\left\langle U(t), U(t) \right\rangle -I_N\right).
			\end{aligned}
		\end{equation*}
		Since $U(t) \in \mathcal{M}_{\geqslant}^{N}$, we have $\langle U(t), U(t) \rangle \geqslant I_N$. Hence $\operatorname{tr}(\langle U(t), U(t) \rangle - I_N) \geqslant 0$, and $\mathrm{I} \leqslant 0$. Moreover, $\langle \mathcal{G} U(t), U(t) \rangle$ is symmetric positive definite, so the leading term is non-positive. We obtain
		\begin{equation}\label{eq:energy time derivative}
			\begin{aligned}
					\frac{\mathrm{d}}{\mathrm{d}t} E(U(t)) &= - \left( \nabla_G E_{\mathcal{G}}(U) \langle \mathcal{G} U, U \rangle^{-1}, \nabla_G E_{\mathcal{G}}(U) \right)_a -  \operatorname{tr}\left(\left\langle U(t), U(t) \right\rangle -I_N\right)
				\\ &\leqslant - \left( \nabla_G E_{\mathcal{G}}(U) \langle \mathcal{G} U, U \rangle^{-1}, \nabla_G E_{\mathcal{G}}(U) \right)_a \leqslant 0.
			\end{aligned}
		\end{equation}
	\end{proof}
	
	\subsection{Well-posedness}
	
	We next establish global well-posedness of
	\eqref{equ: evolution problem}, starting with local Lipschitz
	continuity of $\nabla_GE_{\mathcal G}$.
	
	\begin{lemma}\label{lem: Local Lipschitz property}
		Let $M > 0$ be a constant. If $U, V \in \big(H_0^1(\Omega)\big)^N$ satisfy $\|U\|_a \leqslant M$ and $\|V\|_a \leqslant M$, then there exists a constant $C_M > 0$, depending only on $M$, such that
		\begin{align*}
			\left\| \nabla_G E_{\mathcal{G}}(U) - \nabla_G E_{\mathcal{G}}(V) \right\|_a \leqslant C_M \|U - V\|_a.
		\end{align*}
	\end{lemma}
\begin{proof}
	The proof is given in Appendix \ref{app:lemma-proofs}.
\end{proof}
	
	The Picard--Lindel\"of theorem now gives local well-posedness.
	
	\begin{lemma}\label{lem:local existence}
		For any initial data $U_0 \in \big(H_0^1(\Omega)\big)^N$, there exists a maximal time $T_{\max} \in (0,\infty]$ such that the problem \eqref{equ: evolution problem} admits a unique solution
		\begin{align*}
			U \in C^1([0, T_{\max}); \big(H_0^1(\Omega)\big)^N).
		\end{align*}
	\end{lemma}
	\begin{proof}
		The proof is given in Appendix \ref{app:lemma-proofs}.
	\end{proof}
	
	Energy decay extends the local solution globally. The proof follows the
	continuation argument in \cite[Theorem 3.2]{henning2020sobolev}.
	
	\begin{theorem}\label{thm:global well-posedness}
		If the initial data $U_0 \in \mathcal{M}_{\geqslant}^N$, then \eqref{equ: evolution problem} admits a unique global solution
		\begin{align*}
			U \in C^1([0, \infty); \big(H_0^1(\Omega)\big)^N).
		\end{align*}
	\end{theorem}
	\begin{proof}
		Suppose that the maximal existence time $T$ is finite. For
		$0\leqslant t_1\leqslant t_2<T$, the Cauchy--Schwarz inequality gives
		\begin{align*}
			\| U(t_2) - U(t_1)\|_a^2 & = \left\| \int_{t_1}^{t_2}  \frac{\mathrm{d}  U(t)  }{\mathrm{d}  t} \,\mathrm{d}t\right\|_a^2 \leqslant \left( \int_{t_1}^{t_2} \left\| \frac{\mathrm{d}  U(t)  }{\mathrm{d}  t}\right\|_a \,\mathrm{d}t \right)^2 \\
			& \leqslant (t_2 - t_1) \int_{t_1}^{t_2} \left\| \frac{\mathrm{d}  U(t)  }{\mathrm{d}  t}\right\|_a^2 \,\mathrm{d}t. 
		\end{align*}
		The explicit representation in Theorem \ref{expression of solution} implies $\lambda_{\max}(\langle U(t),U(t)\rangle) \leqslant \lambda_{\max}(\langle U_0,U_0\rangle)$. Hence the spectral bounds of the inner product matrix give
		\begin{equation*}
			\begin{aligned}
				\left\langle U(t), \mathcal{G}U(t) \right\rangle^{-1} \geqslant \frac{\lambda_1}{\lambda_{\max}\left(\left\langle U(t), U(t) \right\rangle\right)}I_N \geqslant \frac{\lambda_1}{\lambda_{\max}\left(\left\langle U_0, U_0 \right\rangle\right)}I_N.
			\end{aligned}
		\end{equation*}
		By \eqref{eq:energy time derivative},
		\begin{equation*}
			\begin{aligned}
				\frac{\mathrm{d}   E(U(t)) }{\mathrm{d}  t}    \leqslant &- \left(\nabla_G E_{\mathcal{G}}(U(t))  \langle \mathcal{G} U(t),  U(t) \rangle^{-1}, \nabla_G E_{\mathcal{G}}(U(t)) \right)_a 
				\\ \leqslant &-\frac{\lambda_1}{\lambda_{\max}\left(\left\langle U_0, U_0 \right\rangle\right)}
				\left(\frac{\mathrm{d}  U(t)  }{\mathrm{d}  t} , \frac{\mathrm{d}  U(t)  }{\mathrm{d}  t}\right)_a.
			\end{aligned}
		\end{equation*}
		Integration over $[t_1,t_2]$ gives
		\begin{equation*}
			\begin{aligned}
				\|U(t_2)-U(t_1)\|_a^2	& \leqslant  (t_2 - t_1) \frac{\lambda_{\max}\left(\left\langle U_0, U_0 \right\rangle\right)}{\lambda_1} \int_{t_1}^{t_2} \left(- 	\frac{\mathrm{d}   E(U(t)) }{\mathrm{d}  t}    \right)\mathrm{d}t  
				\\& =   (t_2 - t_1) \frac{\lambda_{\max}\left(\left\langle U_0, U_0 \right\rangle\right)}{\lambda_1} \big(E(U(t_1)) - E(U(t_2))\big).
			\end{aligned}
		\end{equation*}
		This estimate shows that $U(t)$ is a Cauchy family in $\big(H_0^1(\Omega)\big)^N$ as $t \to T$. Hence there exists $U_T \in \big(H_0^1(\Omega)\big)^N$ such that $U(t) \to U_T$ strongly as $t \to T$. By the local well-posedness result, the same equation with initial value $U_T$ has a solution on $[T,T+\delta)$ for some $\delta>0$, contradicting the maximality of $T$. Therefore, the maximal existence time is infinite.
	\end{proof}

 \subsection{Asymptotic behavior}
 
 With global well-posedness established, we next study the behavior of $U(t)$
 as $t\to\infty$.
 
 \subsubsection{Asymptotic orthogonality}
 
 We first bound the spectrum of $\langle U,\mathcal GU\rangle$. This bound
 yields the decay estimate for the orthogonality error.
 
 \begin{lemma}\label{lem:bound of UGU}
 For any $U \in \mathcal{M}_{\geqslant}^{N}$, the eigenvalues of the matrix $\langle U, \mathcal{G}U \rangle$ satisfy
 	\begin{equation*}
 		\frac{1}{2E(U)} \leqslant \lambda\left(\langle U, \mathcal{G}U \rangle\right) \leqslant \frac{\lambda_{\max}(\langle U, U \rangle)}{\lambda_1}.
 	\end{equation*}
 \end{lemma}
\begin{proof}
	The proof is given in Appendix \ref{app:lemma-proofs}.
\end{proof}

 \begin{theorem}\label{thm: asymptotic orthogonality}
 	Let $U(t)$ be the solution of \eqref{equ: evolution problem} with initial data $U_0 \in \mathcal{M}_{\geqslant}^{N}$. Then
 	\begin{equation}\label{Quasi-orthogonality}
 		\left\|\langle U(t), U(t) \rangle - I_N\right\| \leqslant \left\|\langle U_0, U_0 \rangle - I_N\right\| \exp \left(- \frac{1}{E(U_0)}t\right), \quad \forall\, t \geqslant 0.
 	\end{equation}
 	Consequently,
 	\begin{equation*}
 		\lim_{t \to \infty} \langle U(t), U(t) \rangle = I_N.
 	\end{equation*}
 \end{theorem}
 \begin{proof}
 	Differentiating the squared orthogonality error and using
	\eqref{equ: evolution problem}, we obtain
 	\begin{equation}\label{eq:orthogonality time derivative}
 		\begin{aligned}
 			\frac{\mathrm{d}    }{\mathrm{d}  t} \left\|\left\langle U(t), U(t) \right\rangle  - I_N\right\|^2=& 2 \operatorname{tr}\left[\left(\left\langle U(t),U(t) \right\rangle  - I_N\right) \cdot \left(\left\langle \frac{\mathrm{d}  U(t)  }{\mathrm{d}  t} , U(t) \right\rangle +\left\langle U(t), \frac{\mathrm{d} U(t)    }{\mathrm{d}  t}  \right\rangle  \right)\right]
 			\\=& -4 \operatorname{tr}\left[\left(\left\langle U(t),U(t) \right\rangle  - I_N\right) \cdot \left\langle U(t), \mathcal{G}U(t) \right\rangle  \cdot \left(\left\langle U(t),U(t) \right\rangle  - I_N\right)  \right].
 		\end{aligned}
 	\end{equation}
	Lemma \ref{lem:bound of UGU} gives
 	\begin{equation*}
 		\lambda_{\min}(\langle U(t), \mathcal{G}U(t) \rangle) \geqslant \frac{1}{2E(U(t))}.
 	\end{equation*}
	Since Theorem \ref{thm:energy decay} gives
	$E(U(t))\leqslant E(U_0)$,
 	\begin{equation*}
 		\frac{1}{2E(U(t))} \geqslant \frac{1}{2E(U_0)}.
 	\end{equation*}
	Therefore,
 	\begin{equation*}
 		\frac{\mathrm{d}    }{\mathrm{d}  t} \left\|\left\langle U(t), U(t) \right\rangle  - I_N\right\|^2
 		\leqslant - \frac{2}{E(U_0)}  \left\|\left\langle U(t), U(t) \right\rangle  - I_N\right\|^2.
 	\end{equation*}
	Gr\"onwall's inequality yields
 	\begin{equation*}
 		\left\|\left\langle U(t), U(t) \right\rangle - I_N\right\|^2\leqslant \left\|\left\langle U_0, U_0 \right\rangle -I_N\right\|^2  \exp\left(-  \frac{2}{E(U_0)}t\right), \quad \forall t \geqslant 0.
 	\end{equation*}
 	Taking the square root of both sides completes the proof.
 \end{proof}
 
 Thus the orthogonality error tends to zero without explicit
 orthogonalization.
 
\subsubsection{Convergence}

After establishing asymptotic orthogonality, we study convergence of
$U(t)$. Define the
quasi-Grassmann set
\begin{equation*}
	\mathcal{G}_{\geqslant}^N \coloneqq   \left\{[U]: U \in \mathcal{M}_{\geqslant}^N\right\}.
\end{equation*} 
We also define the restricted set
\begin{equation*} 
	\mathcal{S} \coloneqq \left\{ U \in \big(H_0^1(\Omega)\big)^N : [U] \in \mathcal{G}_{\geqslant}^{N} \cap \mathcal{L}_{E_1} \right\}, 
\end{equation*}
where $\mathcal{L}_{E_1}$ denotes the energy sublevel set given by
\begin{equation*}
	\mathcal{L}_{E_1} \coloneqq \left\{ [U]: U \in \big(H_0^1(\Omega)\big)^N,\ E(U) \leqslant E_1 \coloneqq \frac{E(V_{*}) + E_{\text{ES}}}{2} \right\},
\end{equation*}
and
$E_{\text{ES}} \coloneqq \frac{1}{2}\left(\sum_{i=1}^{N-1}\lambda_i
+\lambda_{N+1}\right)$ is the energy level obtained by replacing
$\lambda_N$ with $\lambda_{N+1}$.
Every critical point in $\mathcal S$ belongs to $[V_*]$.

\begin{theorem}\label{thm: convergence}
	Let $U(t)$ be the solution of \eqref{equ: evolution problem} with initial data $U_0 \in \mathcal{S}$. Then
	\begin{equation*}
		\begin{aligned}
			&	\lim_{t \to \infty} \|[U(t)] - [V_*]\|_a = 0.
		\\	&\lim_{t \to \infty} E(U(t)) = E(V_{*}),
			\\
			&\lim_{t \to \infty} \|\nabla_G E_{\mathcal{G}}(U(t))\|_a = 0.
		\end{aligned}
	\end{equation*}
\end{theorem}
\begin{proof}
	The proof of Theorem \ref{thm:global well-posedness} gives
	\begin{equation*}
		\left\| \frac{\mathrm{d}U(t)}{\mathrm{d}t} \right\|_a^2 \leqslant - \frac{\lambda_{\max}(\langle U_0, U_0 \rangle)}{\lambda_1} \frac{\mathrm{d}}{\mathrm{d}t} E(U(t)).
	\end{equation*}
	Integration over $[0,\infty)$ and the lower bound $E(V_*)$ give
	\begin{equation*}
		\int_{0}^{\infty} \left\| \frac{\mathrm{d}U(t)}{\mathrm{d}t} \right\|_a^2 \mathrm{d}t \leqslant \frac{\lambda_{\max}(\langle U_0, U_0 \rangle)}{\lambda_1} (E(U_0) - E(V_*)) < \infty.
	\end{equation*}
	Since $U(t)$ stays in the bounded set $\mathcal{S}$ and $\nabla_G E_{\mathcal{G}}$ is locally Lipschitz on bounded sets, $\frac{\mathrm{d}U}{\mathrm{d}t}=\nabla_G E_{\mathcal{G}}(U(t))$ is uniformly continuous in $\big(H_0^1(\Omega)\big)^N$. Therefore,
	\begin{equation}\label{eq: gradient sequence limit}
		\lim_{t \to \infty} \left\| \nabla_G E_{\mathcal{G}}(U(t)) \right\|_a = \lim_{t \to \infty} \left\| \frac{\mathrm{d}U(t)}{\mathrm{d}t} \right\|_a = 0. 
	\end{equation}
	
	Suppose that the Grassmann distance does not converge to zero. Then there exist $\delta>0$ and a sequence $t_n\to\infty$ such that
	\begin{equation*}
		\|[U(t_n)]-[V_*]\|_a\geqslant\delta,
		\qquad n\geqslant0.
	\end{equation*}
	Boundedness of $\{U(t_n)\}$ in $\big(H_0^1(\Omega)\big)^N$ gives a subsequence and a limit $\bar U$ such that
	\begin{equation*}
		U(t_n)\rightharpoonup\bar U
		\quad\text{weakly in }\big(H_0^1(\Omega)\big)^N,
		\qquad
		U(t_n)\to\bar U
		\quad\text{strongly in }\big(L^2(\Omega)\big)^N.
	\end{equation*}
	The strong $L^2$ convergence and quasi-orthogonality give $\langle\bar U,\bar U\rangle\geqslant I_N$. Hence $\langle\bar U,\mathcal G\bar U\rangle$ is symmetric positive definite. Consequently,
	\begin{equation*}
		\langle U(t_n),\mathcal GU(t_n)\rangle
		\to\langle\bar U,\mathcal G\bar U\rangle,
		\qquad
		\langle U(t_n),\mathcal GU(t_n)\rangle^{-1}
		\to\langle\bar U,\mathcal G\bar U\rangle^{-1}.
	\end{equation*}
	Using
	\begin{equation*}
		U(t_n)
		=\left(\mathcal GU(t_n)-\nabla_GE_{\mathcal G}(U(t_n))\right)
		\langle U(t_n),\mathcal GU(t_n)\rangle^{-1},
	\end{equation*}
	the boundedness of $\mathcal G:L^2(\Omega)\to H_0^1(\Omega)$ and \eqref{eq: gradient sequence limit} show that $U(t_n)\to\bar U$ strongly in $\big(H_0^1(\Omega)\big)^N$. Passing to the limit gives
	\begin{equation*}
		\nabla_GE_{\mathcal G}(\bar U)=0.
	\end{equation*}
	Since $\mathcal S$ is closed in $\big(H_0^1(\Omega)\big)^N$, we have $\bar U\in\mathcal S$. The uniqueness of the critical class in $\mathcal S$ gives $[\bar U]=[V_*]$. This contradicts the lower bound on $\|[U(t_n)]-[V_*]\|_a$. Therefore,
	\begin{equation*}
		\begin{aligned}
			\lim_{t \rightarrow\infty}\left\|[U(t)]-[V_{*}]\right\|_a=0.
		\end{aligned}
	\end{equation*}
	Orthogonal invariance and continuity of the energy yield
	\begin{equation*}
		\lim_{t\to\infty}E(U(t))=E(V_*).
	\end{equation*}
	The gradient limit is \eqref{eq: gradient sequence limit}.
\end{proof}

\subsubsection{Exponential convergence}

Using the preceding convergence result, we establish the following
asymptotic estimates to derive the exponential convergence rate.
\begin{lemma}\label{lem:continuous asymptotic estimates}
	Let $U(t)$ be the solution of \eqref{equ: evolution problem} with initial
	data $U_0\in\mathcal S$. For any $\varepsilon$ satisfying
	$0<\varepsilon<\frac{1}{\lambda_N}-\frac{1}{\lambda_{N+1}}$, there
	exist $T_{\varepsilon}>0$ and $C_{\varepsilon}>0$ such that, for all
	$t\geqslant T_{\varepsilon}$,
	\begin{equation*}
		\left\|\langle U(t),U(t)\rangle-I_N\right\|
		\leqslant C_{\varepsilon}
		\exp\left(-2\left(\frac{1}{\lambda_N}
		-\frac{\varepsilon}{4}\right)t\right),
	\end{equation*}
	and
	\begin{equation*}
		\left\|\frac{\mathrm dU(t)}{\mathrm dt}\right\|^2
		\leqslant
		\left(\frac{1}{\lambda_{N+1}}+\frac{\varepsilon}{4}\right)
		\left\|\frac{\mathrm dU(t)}{\mathrm dt}\right\|_a^2
		+C_{\varepsilon}
		\exp\left(-4\left(\frac{1}{\lambda_N}
		-\frac{\varepsilon}{4}\right)t\right).
	\end{equation*}
\end{lemma}
\begin{proof}
	The proof is given in Appendix \ref{app:lemma-proofs}.
\end{proof}

The preceding estimates yield the following convergence rates.
\begin{theorem}\label{thm: exponential convergence}
	Let $U(t)$ be the solution of \eqref{equ: evolution problem} with initial data $U_0\in\mathcal S$. Fix $\varepsilon$ such that $0<\varepsilon<\frac{1}{\lambda_N}-\frac{1}{\lambda_{N+1}}$.
	Then there exist $T_{\varepsilon}>0$, $C_{\varepsilon}>0$,
	and $Q_*\in\mathcal O^N$ such that for all $t\geqslant T_{\varepsilon}$,
	\begin{equation*}
		\begin{aligned}
			&\left\| \nabla_G E_{\mathcal{G}}(U(t)) \right\|_a \leqslant C_{\varepsilon}\exp\left(-\left( \frac{1}{\lambda_N} - \frac{1}{\lambda_{N+1}}-\varepsilon \right)t\right),
			\\
			&\left\|U(t) - V_*Q_*\right\|_a \leqslant C_{\varepsilon}\exp\left(-\left( \frac{1}{\lambda_N} - \frac{1}{\lambda_{N+1}}-\varepsilon \right)t\right),
			\\
			&E(U(t)) - E(V_*) \leqslant C_{\varepsilon}\exp\left(-2\left( \frac{1}{\lambda_N} - \frac{1}{\lambda_{N+1}}-\varepsilon \right)t\right).
		\end{aligned}
	\end{equation*}
\end{theorem}

\begin{proof}
	We first prove the gradient estimate using Lemma
	\ref{lem:continuous asymptotic estimates}.
	The map
	$U\mapsto\mathcal GU-U\langle\mathcal GU,U\rangle$ is continuously
	Fr\'echet differentiable. Hence
	$U\in C^2([0,\infty);\big(H_0^1(\Omega)\big)^N)$.
	Differentiating gives
	\begin{equation}\label{eq: continuous gradient derivative}
		\begin{aligned}
			\frac{1}{2}\frac{\mathrm{d}}{\mathrm{d}t}
			\left\|\nabla_GE_{\mathcal G}(U(t))\right\|_a^2
			={}&
			\left(
			\frac{\mathrm{d}U(t)}{\mathrm{d}t},
			\frac{\mathrm{d}U(t)}{\mathrm{d}t}
			\right)-
			\left(
			\frac{\mathrm{d}U(t)}{\mathrm{d}t},
			\frac{\mathrm{d}U(t)}{\mathrm{d}t}
			\langle U(t),\mathcal GU(t)\rangle
			\right)_a\\
			&-
			\left(
			\frac{\mathrm{d}U(t)}{\mathrm{d}t},
			U(t)\frac{\mathrm{d}}{\mathrm{d}t}
			\langle U(t),\mathcal GU(t)\rangle
			\right)_a.
		\end{aligned}
	\end{equation}

	To estimate the last term, the evolution equation and the identity
	$(\mathcal GV,W)_a=(V,W)$ give
	$$	U(t)
	=
	\left(
	\mathcal GU(t)-\frac{\mathrm{d}U(t)}{\mathrm{d}t}
	\right)
	\langle U(t),\mathcal GU(t)\rangle^{-1},$$
	and thus
	\begin{equation*}
		\begin{aligned}
			\left\langle
			\frac{\mathrm{d}U(t)}{\mathrm{d}t},U(t)
			\right\rangle_a
			={}&
			-\langle U(t),\mathcal GU(t)\rangle
			\left(\langle U(t),U(t)\rangle-I_N\right)
			\langle U(t),\mathcal GU(t)\rangle^{-1}\\
			&-
			\left\langle
			\frac{\mathrm{d}U(t)}{\mathrm{d}t},
			\frac{\mathrm{d}U(t)}{\mathrm{d}t}
			\right\rangle_a
			\langle U(t),\mathcal GU(t)\rangle^{-1}.
		\end{aligned}
	\end{equation*}
	Consequently,
	\begin{equation*}
		\begin{aligned}
			&\left(
			\frac{\mathrm{d}U(t)}{\mathrm{d}t},
			U(t)\frac{\mathrm{d}}{\mathrm{d}t}
			\langle U(t),\mathcal GU(t)\rangle
			\right)_a\\
			=&
			-\operatorname{tr}\left(
			\langle U(t),\mathcal GU(t)\rangle
			\left(\langle U(t),U(t)\rangle-I_N\right)
			\langle U(t),\mathcal GU(t)\rangle^{-1}
			\frac{\mathrm{d}}{\mathrm{d}t}
			\langle U(t),\mathcal GU(t)\rangle
			\right)\\
			&-
			\operatorname{tr}\left(
			\left\langle
			\frac{\mathrm{d}U(t)}{\mathrm{d}t},
			\frac{\mathrm{d}U(t)}{\mathrm{d}t}
			\right\rangle_a
			\langle U(t),\mathcal GU(t)\rangle^{-1}
			\frac{\mathrm{d}}{\mathrm{d}t}
			\langle U(t),\mathcal GU(t)\rangle
			\right).
		\end{aligned}
	\end{equation*}
	Energy decay and Lemma \ref{lem:bound of UGU} imply
	\begin{equation*}
		\begin{aligned}
		&	\langle U(t),\mathcal GU(t)\rangle^{-1}
			\leqslant2E(U_0)I_N,\\
			&\langle U(t),\mathcal GU(t)\rangle
			\leqslant
			\frac{\lambda_{\max}(\langle U(t),U(t)\rangle)}
			{\lambda_1}I_N
			\leqslant
			\frac{2c_\Omega^2E(U_0)}{\lambda_1}I_N.
		\end{aligned}
	\end{equation*}
	Also,
	\begin{equation*}
		\left\|
		\frac{\mathrm{d}}{\mathrm{d}t}
		\langle U(t),\mathcal GU(t)\rangle
		\right\|
		\leqslant
		C\left\|\frac{\mathrm{d}U(t)}{\mathrm{d}t}\right\|_a.
	\end{equation*}
	Hence the absolute value of the last term in
	\eqref{eq: continuous gradient derivative} is bounded by
	\begin{equation*}
		C\left\|\langle U(t),U(t)\rangle-I_N\right\|
		\left\|\frac{\mathrm{d}U(t)}{\mathrm{d}t}\right\|_a
		+C\left\|\frac{\mathrm{d}U(t)}{\mathrm{d}t}\right\|_a^3.
	\end{equation*}
	Young's inequality,
	\eqref{eq: gradient sequence limit}, and the orthogonality estimate
	therefore yield, after one more increase of $T_{\varepsilon}$,
	\begin{equation*}
		\begin{aligned}
			\left|
			\left(
			\frac{\mathrm{d}U(t)}{\mathrm{d}t},
			U(t)\frac{\mathrm{d}}{\mathrm{d}t}
			\langle U(t),\mathcal GU(t)\rangle
			\right)_a
			\right|\leqslant
			\frac{\varepsilon}{2}
			\left\|\frac{\mathrm{d}U(t)}{\mathrm{d}t}\right\|_a^2
			+C_{\varepsilon}
			\exp\left(
			-4\left(\frac{1}{\lambda_N}
			-\frac{\varepsilon}{4}\right)t
			\right).
		\end{aligned}
	\end{equation*}

	By \eqref{eq: continuous asymptotic spectral separation}, both
	$\langle U(t),\mathcal GU(t)\rangle
	-(\lambda_N^{-1}-\varepsilon/4)I_N$ and
	$\langle\frac{\mathrm{d}U(t)}{\mathrm{d}t},
	\frac{\mathrm{d}U(t)}{\mathrm{d}t}\rangle_a$
	are positive semidefinite. The trace of their product is
	nonnegative, so
	\begin{equation*}
		\left(
		\frac{\mathrm{d}U(t)}{\mathrm{d}t},
		\frac{\mathrm{d}U(t)}{\mathrm{d}t}
		\langle U(t),\mathcal GU(t)\rangle
		\right)_a
		\geqslant
		\left(\frac{1}{\lambda_N}-\frac{\varepsilon}{4}\right)
		\left\|\frac{\mathrm{d}U(t)}{\mathrm{d}t}\right\|_a^2.
	\end{equation*}
	Combining the preceding estimates in
	\eqref{eq: continuous gradient derivative} gives
	\begin{equation*}
		\begin{aligned}
			\frac{\mathrm{d}}{\mathrm{d}t}
			\left\|\nabla_GE_{\mathcal G}(U(t))\right\|_a^2
			\leqslant{}&
			-2\left(
			\frac{1}{\lambda_N}
			-\frac{1}{\lambda_{N+1}}
			-\varepsilon
			\right)
			\left\|\nabla_GE_{\mathcal G}(U(t))\right\|_a^2\\
			&+
			C_{\varepsilon}
			\exp\left(
			-4\left(\frac{1}{\lambda_N}
			-\frac{\varepsilon}{4}\right)t
			\right).
		\end{aligned}
	\end{equation*}

	Since
	\begin{equation*}
		4\left(\frac{1}{\lambda_N}-\frac{\varepsilon}{4}\right)
		-2\left(\frac{1}{\lambda_N}
		-\frac{1}{\lambda_{N+1}}-\varepsilon\right)
		=\frac{2}{\lambda_N}+\frac{2}{\lambda_{N+1}}
		+\varepsilon>0,
	\end{equation*}
	integration yields
	\begin{equation*}
		\left\|\nabla_GE_{\mathcal G}(U(t))\right\|_a^2
		\leqslant
		C_{\varepsilon}
		\exp\left(
		-2\left(
		\frac{1}{\lambda_N}
		-\frac{1}{\lambda_{N+1}}
		-\varepsilon
		\right)t
		\right).
	\end{equation*}
	Thus,
	\begin{equation*}
		\left\|\nabla_GE_{\mathcal G}(U(t))\right\|_a
		\leqslant
		C_{\varepsilon}
		\exp\left(
		-\left(
		\frac{1}{\lambda_N}
		-\frac{1}{\lambda_{N+1}}
		-\varepsilon
		\right)t
		\right).
	\end{equation*}

	This proves the first estimate of the theorem. For
	$t_2\geqslant t_1\geqslant T_{\varepsilon}$,
	\begin{equation*}
		\begin{aligned}
			\|U(t_2)-U(t_1)\|_a
			&\leqslant
			\int_{t_1}^{t_2}
			\left\|\frac{\mathrm{d}U(s)}{\mathrm{d}s}\right\|_a
			\,\mathrm{d}s\\
			&\leqslant
			C_{\varepsilon}
			\exp\left(
			-\left(
			\frac{1}{\lambda_N}
			-\frac{1}{\lambda_{N+1}}
			-\varepsilon
			\right)t_1
			\right).
		\end{aligned}
	\end{equation*}
	Hence $U(t)$ is Cauchy in $\big(H_0^1(\Omega)\big)^N$ and converges
	to a limit. Theorem \ref{thm: convergence} shows that this limit is
	$V_*Q_*$ for some $Q_*\in\mathcal O^N$. Letting $t_1=t$ and
	$t_2\to\infty$ gives
	\begin{equation*}
		\left\|U(t)-V_*Q_*\right\|_a
		\leqslant
		C_{\varepsilon}
		\exp\left(
		-\left(
		\frac{1}{\lambda_N}
		-\frac{1}{\lambda_{N+1}}
		-\varepsilon
		\right)t
		\right).
	\end{equation*}

	Finally, \eqref{eq:energy time derivative},
	$\langle U(t),U(t)\rangle\geqslant I_N$, and the uniform inverse
	bound imply
	\begin{equation*}
		\begin{aligned}
			\left|\frac{\mathrm{d}E(U(t))}{\mathrm{d}t}\right|
			\leqslant{}&
			C\left\|\nabla_GE_{\mathcal G}(U(t))\right\|_a^2
			+\sqrt N
			\left\|\langle U(t),U(t)\rangle-I_N\right\|.
		\end{aligned}
	\end{equation*}
	Since $E(U(t))\to E(V_*)$, integration over $[t,\infty)$ gives
	\begin{equation*}
		\begin{aligned}
			E(U(t))-E(V_*)
			\leqslant{}&
			C\int_t^\infty
			\left\|\nabla_GE_{\mathcal G}(U(s))\right\|_a^2
			\,\mathrm{d}s+
			\sqrt N\int_t^\infty
			\left\|\langle U(s),U(s)\rangle-I_N\right\|
			\,\mathrm{d}s.
		\end{aligned}
	\end{equation*}
	The orthogonality term decays faster than the squared gradient
	term because
	$2(\lambda_N^{-1}-\varepsilon/4)>
	2(\lambda_N^{-1}-\lambda_{N+1}^{-1}-\varepsilon)$.
	Therefore,
	\begin{equation*}
		E(U(t))-E(V_*)
		\leqslant
		C_{\varepsilon}
		\exp\left(
		-2\left(
		\frac{1}{\lambda_N}
		-\frac{1}{\lambda_{N+1}}
		-\varepsilon
		\right)t
		\right).
	\end{equation*}
\end{proof}

 \section{Time discretization}\label{section:Time discretization}
 
 To retain the quasi-orthogonal structure of the continuous evolution,
 we construct a two-stage time discretization of
 \eqref{equ: evolution problem}.
 We analyze its well-posedness, energy decay, and long-time behavior.
 
 \subsection{A quasi-orthogonal scheme}
 
 Let $\{t_n\}_{n=0}^{\infty} \subset [0, +\infty)$ be a strictly increasing sequence of discrete time levels satisfying $t_0 = 0$ and $\lim_{n \to \infty} t_n = +\infty$. We define the variable time step size as
 \begin{equation*}
 	\tau_n \coloneqq t_{n+1} - t_n.
 \end{equation*}
 
 Given $U_n\in\mathcal M_{\geqslant}^N$, compute $U_{n+1}$ from
 \begin{equation}\label{equ: numerical scheme}
 	\left\{\begin{aligned}
 		&\frac{\hat{U}_{n+1} - U_n}{\tau_n} = \mathcal{G} U_n \langle U_n, \tilde{U}_{n + \frac{1}{2}} \rangle - U_n \langle \mathcal{G} U_n, \tilde{U}_{n + \frac{1}{2}} \rangle, \quad \text{where } \tilde{U}_{n + \frac{1}{2}} \coloneqq \frac{\hat{U}_{n+1} + U_n}{2},
 		\\
	 	&\frac{U_{n+1} - \hat{U}_{n+1}}{\tau_n} = - \mathcal{G}U_{n+1} \left( \langle U_{n+1}, U_{n+1} \rangle - I_N \right).
 	\end{aligned}\right.
 \end{equation}
 
 For fixed $U$, define
 $\mathcal A_U:\big(H_0^1(\Omega)\big)^N\to
 \big(H_0^1(\Omega)\big)^N$ by
 \begin{equation*}
 	\mathcal{A}_U V \coloneqq \mathcal{G}U \langle U, V \rangle - U \langle U, \mathcal{G}V \rangle.
 \end{equation*}
 It satisfies
 \begin{equation*}
 	\langle V, \mathcal{A}_U V \rangle = - \langle \mathcal{A}_U V, V \rangle, \quad \forall\, V \in \big(H_0^1(\Omega)\big)^N.
 \end{equation*}
The vector field admits the decomposition
\begin{equation}\label{eq: vector field splitting}
	\nabla_GE_{\mathcal G}(U)
	=\mathcal A_UU
	-\mathcal GU\left(\langle U,U\rangle-I_N\right).
\end{equation}
The two stages treat these two terms separately. Thus,
\eqref{equ: numerical scheme} can be written as
 \begin{equation}
 	\tag{\ref{equ: numerical scheme}$*$}
 	\left\{\begin{aligned}
 		&\hat{U}_{n+1} =U_n +\tau_n \mathcal{A}_{U_n} \tilde{U}_{n + \frac12}, 
	 	\\&U_{n+1} = \hat{U}_{n+1} -\tau_n \mathcal{G}U_{n+1}\left(\langle U_{n+1},U_{n+1}\rangle -I_N \right).
 	\end{aligned}\right.
 \end{equation}
 
 We use the matrices
 \begin{align*}
 	O_n &\coloneqq \langle U_n, U_n \rangle - I_N,
 	\\
 	A_n &\coloneqq \langle U_n, \mathcal{G}U_n \rangle,
 	\\
 	B_n &\coloneqq \langle \mathcal{G}U_n, \mathcal{G}U_n \rangle.
 \end{align*}
 
 They satisfy the following bounds.
 
 \begin{lemma}\label{lem: matrix bounds}
 	For any $U_n \in \big(H_0^1(\Omega)\big)^N$, the matrices $A_n$ and $B_n$ are positive semidefinite and satisfy
 	\begin{equation*}
 		0 \leqslant A_n \leqslant \frac{1}{\lambda_1}\langle U_n,U_n\rangle,
 		\qquad
 		0 \leqslant B_n \leqslant \frac{1}{\lambda_1}A_n.
 	\end{equation*}
 	If $U_n \in \mathcal{M}_{\geqslant}^N$, then $O_n$ is positive semidefinite, $A_n$ and $B_n$ are positive definite, and
 	\begin{equation*}
 		0 \leqslant O_n \leqslant \langle U_n,U_n\rangle,
 		\qquad
 		A_n \leqslant \frac{\lambda_{\max}(\langle U_n,U_n\rangle)}{\lambda_1}I_N,
 		\qquad
 		B_n \leqslant \frac{\lambda_{\max}(\langle U_n,U_n\rangle)}{\lambda_1^2}I_N.
 	\end{equation*}
 \end{lemma}
\begin{proof}
	The proof is given in Appendix \ref{app:lemma-proofs}.
\end{proof}
 The next theorem gives well-posedness and preservation of the
 quasi-Stiefel set.
 \begin{theorem}\label{thm: stability}
 	Suppose that $U_0\in\mathcal{M}_{\geqslant}^N$ and $0<\tau_n<\lambda_1$ for all $n\geqslant0$. Then the sequence $\{U_n\}_{n=0}^{\infty}$ generated by \eqref{equ: numerical scheme} is well defined and satisfies
 	\begin{equation*}
 		U_n\in\mathcal{M}_{\geqslant}^N,
 		\qquad
 		\lambda_{\max}\big(\langle U_n,U_n\rangle\big)
 		\leqslant
 		\lambda_{\max}\big(\langle U_0,U_0\rangle\big),
 		\qquad \forall\,n\geqslant0.
 	\end{equation*}
	Moreover, the orthogonality error matrices satisfy the Gram matrix relation
	\begin{equation}\label{eq: discrete Gram matrix relation}
		O_n=O_{n+1}
		+\tau_n\left(A_{n+1}O_{n+1}+O_{n+1}A_{n+1}\right)
		+\tau_n^2O_{n+1}B_{n+1}O_{n+1}.
	\end{equation}
 \end{theorem}
 \begin{proof}
	Assume inductively that $U_n\in\mathcal{M}_{\geqslant}^N$.
	The first stage can be written as
	\begin{equation*}
		\left(I-\frac{\tau_n}{2}\mathcal A_{U_n}\right)\hat U_{n+1}
		=
		\left(I+\frac{\tau_n}{2}\mathcal A_{U_n}\right)U_n.
	\end{equation*}
	The range of $\mathcal A_{U_n}$ is contained in the space spanned by
	the components of $U_n$ and $\mathcal GU_n$, and is therefore finite
	dimensional. If
	\begin{equation*}
		\left(I-\frac{\tau_n}{2}\mathcal A_{U_n}\right)V=0,
	\end{equation*}
	then the skew-symmetry of $\mathcal A_{U_n}$ gives
	\begin{equation*}
		\|V\|^2
		=\frac{\tau_n}{2}(V,\mathcal A_{U_n}V)=0.
	\end{equation*}
	Thus, $I-\frac{\tau_n}{2}\mathcal A_{U_n}$ has a trivial kernel.
	Since it is the identity minus a finite-rank operator, the Fredholm
	alternative shows that it is invertible. Hence, $\hat U_{n+1}$ is
	uniquely determined.

 	The skew-symmetry of $\mathcal{A}_{U_n}$ and the midpoint relation imply
 	\begin{equation*}
 		\begin{aligned}
 			\langle\hat{U}_{n+1},\hat{U}_{n+1}\rangle-\langle U_n,U_n\rangle
 			&=\langle\hat{U}_{n+1}-U_n,\tilde{U}_{n+\frac12}\rangle
 			+\langle\tilde{U}_{n+\frac12},\hat{U}_{n+1}-U_n\rangle\\
 			&=\tau_n\left(\langle\mathcal{A}_{U_n}\tilde{U}_{n+\frac12},\tilde{U}_{n+\frac12}\rangle
 			+\langle\tilde{U}_{n+\frac12},\mathcal{A}_{U_n}\tilde{U}_{n+\frac12}\rangle\right)=0.
 		\end{aligned}
 	\end{equation*}
	Thus, $\langle\hat{U}_{n+1},\hat{U}_{n+1}\rangle=\langle U_n,U_n\rangle$.

	The second stage is the Euler--Lagrange equation associated with the functional
 	\begin{equation*}
		V\longmapsto \frac{1}{2\tau_n}\|V-\hat{U}_{n+1}\|_a^2
		+\frac14\operatorname{tr}\left(\left(\langle V,V\rangle-I_N\right)^2\right).
 	\end{equation*}
	Indeed, its first directional derivative at $V$ in the direction $W$ is
	\begin{equation*}
    \begin{aligned}
		\frac{1}{\tau_n}(V-\hat U_{n+1},W)_a
		+\left(V\left(\langle V,V\rangle-I_N\right),W\right)
         =  \left(\frac{V-\hat U_{n+1}}{\tau_n}+ \mathcal GV\left(\langle V,V\rangle-I_N\right),W \right)_a,
    \end{aligned}
	\end{equation*}
	Setting this derivative to zero for every
	$W\in\big(H_0^1(\Omega)\big)^N$ gives the second stage of
	\eqref{equ: numerical scheme}. For any
	$V,W\in\big(H_0^1(\Omega)\big)^N$, the second directional derivative
	at $V$ in the direction $W$ satisfies
 	\begin{equation*}
		\begin{aligned}
			\frac{1}{\tau_n}\|W\|_a^2
			+\frac12\operatorname{tr}\left(\left(\langle V,W\rangle+\langle W,V\rangle\right)^2\right)
			+\operatorname{tr}\left(\left(\langle V,V\rangle-I_N\right)\langle W,W\rangle\right)
		\geqslant \left(\frac{1}{\tau_n}-\frac{1}{\lambda_1}\right)\|W\|_a^2.
		\end{aligned}
 	\end{equation*}
	Since $0<\tau_n<\lambda_1$, the functional is strongly convex. It is
	also coercive and weakly lower semicontinuous. Hence it has a unique
	minimizer, which determines $U_{n+1}$.

	It remains to prove that $O_{n+1}\geqslant0$ and
	$\lambda_{\max}(O_{n+1})\leqslant\lambda_{\max}(O_n)$. The identity
 	\begin{equation*}
		\hat{U}_{n+1}=U_{n+1}+\tau_n\mathcal{G}U_{n+1}O_{n+1}
 	\end{equation*}
	and the preservation of the Gram matrix in the first stage yield
	\eqref{eq: discrete Gram matrix relation}.
	Suppose that $\lambda_{\min}(O_{n+1})<0$, and let $x$ be a
	corresponding unit eigenvector. Since
	$\langle U_{n+1},U_{n+1}\rangle=I_N+O_{n+1}\geqslant0$, we have
	$\lambda_{\min}(O_{n+1})\geqslant-1$. Lemma \ref{lem: matrix bounds}
	also gives
	\begin{equation*}
		x^\top B_{n+1}x
		\leqslant\frac{1}{\lambda_1^2}
		x^\top\langle U_{n+1},U_{n+1}\rangle x
		=\frac{1+\lambda_{\min}(O_{n+1})}{\lambda_1^2}
		\leqslant\frac{1}{\lambda_1^2}.
	\end{equation*}
	Therefore,
 	\begin{equation*}
			x^\top O_nx
			=\lambda_{\min}(O_{n+1})
			\bigl(1+2\tau_nx^\top A_{n+1}x+\tau_n^2\lambda_{\min}(O_{n+1})x^\top B_{n+1}x\bigr).
 	\end{equation*}
	The factor in parentheses satisfies
	\begin{equation*}
		1+2\tau_nx^\top A_{n+1}x
		+\tau_n^2\lambda_{\min}(O_{n+1})x^\top B_{n+1}x
		\geqslant 1-\frac{\tau_n^2}{\lambda_1^2}>0.
	\end{equation*}
	This contradicts $O_n\geqslant0$. Therefore, $O_{n+1}\geqslant0$.

	Let $x$ now be a unit eigenvector corresponding to
	$\lambda_{\max}(O_{n+1})$. Since $O_{n+1}\geqslant0$, the same identity
	gives
 	\begin{equation*}
		\begin{aligned}
			\lambda_{\max}(O_n)
			&\geqslant x^\top O_nx\\
			&=\lambda_{\max}(O_{n+1})
			\left(1+2\tau_nx^\top A_{n+1}x
			+\tau_n^2\lambda_{\max}(O_{n+1})x^\top B_{n+1}x\right)\\
			&\geqslant\lambda_{\max}(O_{n+1}).
		\end{aligned}
 	\end{equation*}
	Since $\langle U_n,U_n\rangle=I_N+O_n$, it follows that
	\begin{equation*}
		\lambda_{\max}\big(\langle U_{n+1},U_{n+1}\rangle\big)
		\leqslant
		\lambda_{\max}\big(\langle U_n,U_n\rangle\big).
	\end{equation*}
	Since $U_0\in\mathcal M_{\geqslant}^N$, this completes the induction.
 \end{proof}
 
 \subsection{Energy decay}
	 
 The stability result controls the iterates in the quasi-Stiefel set. We now
 derive an energy estimate for the discrete update.
 
 \begin{theorem}\label{thm: Energy E(Un) decay}
 	Let
 	\begin{equation*}
 		\begin{aligned}
 			c_a
 			&\coloneqq
 			\frac{N\lambda_{\max}(\langle U_0,U_0\rangle)}{\lambda_1}
 			+\frac{\sqrt{2NE(U_0)\lambda_{\max}(\langle U_0,U_0\rangle)}}{\lambda_1^{3/2}},\\
 			c_e
 			&\coloneqq
 			2\left(
 			E(U_0)c_a
 			+c_\Omega^2E(U_0)\lambda_{\max}(\langle U_0,U_0\rangle)
 			+1
 			\right).
 		\end{aligned}
 	\end{equation*}
 	Suppose that $U_0\in\mathcal{M}_{\geqslant}^N$ and
 	\begin{equation*}
 		0<\tau_n<
 		\min\left\{
 		\lambda_1,
 		\frac{1}{2c_a},
 		\frac{\lambda_1}{2c_e\lambda_{\max}(\langle U_0,U_0\rangle)}
 		\right\},
 		\qquad n\geqslant0.
 	\end{equation*}
 	Then the scheme \eqref{equ: numerical scheme} satisfies
 	\begin{equation*}
 		E(U_n)-E(U_{n+1})
 		\geqslant
 		\left(
 		\frac{\lambda_1}{2\lambda_{\max}(\langle U_0,U_0\rangle)}
 		-c_e\tau_n
 		\right)
 		\tau_n\|\mathcal{A}_{U_n}U_n\|_a^2,
 		\qquad n\geqslant0.
 	\end{equation*}
 \end{theorem}
 
 \begin{proof}
 	We argue by induction. Assume that $E(U_n)\leqslant E(U_0)$. Theorem \ref{thm: stability} gives
 	\begin{equation*}
 		U_n\in\mathcal{M}_{\geqslant}^N,
 		\qquad
 		\lambda_{\max}(\langle U_n,U_n\rangle)
 		\leqslant\lambda_{\max}(\langle U_0,U_0\rangle).
 	\end{equation*}
	For any $V\in\big(H_0^1(\Omega)\big)^N$, the definition of $\mathcal{A}_{U_n}$ and the Cauchy--Schwarz inequality give
	\begin{equation*}
		\begin{aligned}
			\|\mathcal{G}U_n\langle U_n,V\rangle\|_a
			&\leqslant \|\mathcal{G}U_n\|_a\|U_n\|\|V\|
			\leqslant
			\frac{N\lambda_{\max}(\langle U_n,U_n\rangle)}{\lambda_1}\|V\|_a,\\
			\|U_n\langle U_n,\mathcal{G}V\rangle\|_a
			&\leqslant \|U_n\|_a\|U_n\|\|\mathcal{G}V\|
			\leqslant
			\frac{\sqrt{2NE(U_n)\lambda_{\max}(\langle U_n,U_n\rangle)}}{\lambda_1^{3/2}}\|V\|_a.
		\end{aligned}
	\end{equation*}
	Adding these estimates and using the induction hypothesis yields
	\begin{equation*}
		\|\mathcal{A}_{U_n}V\|_a\leqslant c_a\|V\|_a,
		\qquad
		V\in\big(H_0^1(\Omega)\big)^N.
	\end{equation*}

	For the first stage, direct expansion gives
 	\begin{equation}\label{eq: energy difference}
 		\begin{aligned}
 			E(U_n)-E(\hat{U}_{n+1})
 			&=-\tau_n(\mathcal{A}_{U_n}U_n,U_n)_a\\
 			&\quad
 			-\tau_n\left(
 			\mathcal{A}_{U_n}(\tilde{U}_{n+\frac12}-U_n),U_n
 			\right)_a
 			-\frac{\tau_n^2}{2}
 			\|\mathcal{A}_{U_n}\tilde{U}_{n+\frac12}\|_a^2.
 		\end{aligned}
 	\end{equation}

	For the leading term, a direct calculation gives
	\begin{equation*}
		\|\mathcal{A}_{U_n}U_n\|_a^2
		\leqslant
		2\lambda_{\max}(A_n)
		\bigl(-(\mathcal{A}_{U_n}U_n,U_n)_a\bigr).
	\end{equation*}
	Lemma \ref{lem: matrix bounds} gives
	$\lambda_{\max}(A_n)
	\leqslant\lambda_{\max}(\langle U_n,U_n\rangle)/\lambda_1$.
	Hence
 	\begin{equation*}
 		-(\mathcal{A}_{U_n}U_n,U_n)_a
 		\geqslant
 		\frac{\lambda_1}{2\lambda_{\max}(\langle U_n,U_n\rangle)}
 		\|\mathcal{A}_{U_n}U_n\|_a^2.
 	\end{equation*}

	For the remaining inner product in \eqref{eq: energy difference},
	Lemma \ref{lem:bound of UGU} gives
	$A_n^{-1}\leqslant2E(U_n)I_N$, and
	\begin{equation*}
		\begin{aligned}
			\left(
			\mathcal{A}_{U_n}(\tilde{U}_{n+\frac12}-U_n),U_n
			\right)_a
			=&
			-\left(
			\mathcal{A}_{U_n}(\tilde{U}_{n+\frac12}-U_n),
			\mathcal{A}_{U_n}U_nA_n^{-1}
			\right)_a\\
			&
			-\left(
			\tilde{U}_{n+\frac12}-U_n,
			\mathcal{A}_{U_n}U_n
			\langle U_n,U_n\rangle A_n^{-1}
			\right).
		\end{aligned}
	\end{equation*}
	Here we used the skew-symmetry of $\mathcal{A}_{U_n}$ in the last term. Since
	$\tilde{U}_{n+\frac12}-U_n
	=\frac{\tau_n}{2}\mathcal{A}_{U_n}\tilde{U}_{n+\frac12}$,
	we obtain
 	\begin{equation*}
 		\begin{aligned}
 			&\left|
 			\left(
 			\mathcal{A}_{U_n}(\tilde{U}_{n+\frac12}-U_n),U_n
 			\right)_a
 			\right|\\
 			\leqslant &
 			\left(
 			E(U_n)c_a
 			+c_\Omega^2E(U_n)\lambda_{\max}(\langle U_n,U_n\rangle)
 			\right)
 			\tau_n
 			\|\mathcal{A}_{U_n}\tilde{U}_{n+\frac12}\|_a
 			\|\mathcal{A}_{U_n}U_n\|_a.
 		\end{aligned}
 	\end{equation*}

	Moreover,
 	\begin{equation*}
		\left\|
		\mathcal{A}_{U_n}\tilde{U}_{n+\frac12}
		-\mathcal{A}_{U_n}U_n
		\right\|_a
		\leqslant
		\frac{c_a\tau_n}{2}
		\|\mathcal{A}_{U_n}\tilde{U}_{n+\frac12}\|_a.
 	\end{equation*}
	Since $c_a\tau_n<1/2$, this yields
 	\begin{equation*}
 		\frac45\|\mathcal{A}_{U_n}U_n\|_a
 		\leqslant
 		\|\mathcal{A}_{U_n}\tilde{U}_{n+\frac12}\|_a
 		\leqslant
 		\frac43\|\mathcal{A}_{U_n}U_n\|_a.
 	\end{equation*}
	Substitution into \eqref{eq: energy difference} gives
 	\begin{equation*}
 		E(U_n)-E(\hat{U}_{n+1})
 		\geqslant
 		\left(
 		\frac{\lambda_1}{2\lambda_{\max}(\langle U_0,U_0\rangle)}
 		-c_e\tau_n
 		\right)
 		\tau_n\|\mathcal{A}_{U_n}U_n\|_a^2.
 	\end{equation*}

	For the second stage, write the scheme as
	\begin{equation*}
		\hat{U}_{n+1}
		=U_{n+1}+\tau_n\mathcal{G}U_{n+1}O_{n+1}.
	\end{equation*}
	Then
 	\begin{equation*}
 		\begin{aligned}
 			E(\hat{U}_{n+1})-E(U_{n+1})
 			=\tau_n\operatorname{tr}\left(
 			\langle U_{n+1},U_{n+1}\rangle O_{n+1}
 			\right)
 			+\frac{\tau_n^2}{2}
 			\|\mathcal{G}U_{n+1}O_{n+1}\|_a^2
 			\geqslant0.
 		\end{aligned}
 	\end{equation*}
	Here $O_{n+1}\geqslant0$ by Theorem \ref{thm: stability}, and
	$\langle U_{n+1},U_{n+1}\rangle=I_N+O_{n+1}$ commutes with $O_{n+1}$.
	Combining the two stages proves the stated estimate and closes the induction.
 \end{proof}
 
 \subsection{Asymptotic behavior}

 The stability and energy estimates provide the bounds needed for the
 long-time analysis of the discrete iterates. We consider orthogonality,
 convergence, and convergence rates in turn.

 \subsubsection{Asymptotic orthogonality}
 
 The Gram matrix relation \eqref{eq: discrete Gram matrix relation} yields a contraction
 of the discrete orthogonality error.

 \begin{theorem}\label{thm: asymptotic orthogonality of Un}
 	Under the assumptions of Theorem \ref{thm: Energy E(Un) decay} and
	$\sum_{n=0}^{\infty}\tau_n=\infty$, the orthogonality error satisfies
 	\begin{equation}\label{eq: orthogonality error relation}
 		\|\langle U_{n+1},U_{n+1}\rangle-I_N\|
 		\leqslant
 		\left(1+\frac{\tau_n}{E(U_0)}\right)^{-1}
 		\|\langle U_n,U_n\rangle-I_N\|.
 	\end{equation}
 	Consequently,
 	\begin{equation*}
 		\lim_{n \to \infty} \langle U_n, U_n \rangle = I_N.
 	\end{equation*}
 \end{theorem}

 \begin{proof}
 	Taking the inner product of
	\eqref{eq: discrete Gram matrix relation} with $O_{n+1}$ gives
 	\begin{equation*}
 		\begin{aligned}
 			\operatorname{tr}(O_nO_{n+1})
 			&=\|O_{n+1}\|^2
 			+2\tau_n\operatorname{tr}(A_{n+1}O_{n+1}^2)
 			+\tau_n^2\operatorname{tr}(B_{n+1}O_{n+1}^3).
 		\end{aligned}
 	\end{equation*}
 	Theorem \ref{thm: stability}, Theorem \ref{thm: Energy E(Un) decay}, and Lemma \ref{lem:bound of UGU} imply
 	\begin{equation*}
 		O_{n+1}\geqslant0,
 		\qquad
 		A_{n+1}\geqslant\frac{1}{2E(U_{n+1})}I_N
 		\geqslant\frac{1}{2E(U_0)}I_N.
 	\end{equation*}
 	Since $B_{n+1}\geqslant0$, it follows that
 	\begin{equation*}
 		\operatorname{tr}(O_nO_{n+1})
 		\geqslant
 		\left(1+\frac{\tau_n}{E(U_0)}\right)
 		\|O_{n+1}\|^2.
 	\end{equation*}
 	The Cauchy--Schwarz inequality gives
 	\begin{equation*}
 		\operatorname{tr}(O_nO_{n+1})
 		\leqslant\|O_n\|\|O_{n+1}\|.
 	\end{equation*}
 	If $O_{n+1}=0$, \eqref{eq: orthogonality error relation} is immediate. Otherwise, division by $\|O_{n+1}\|$ proves \eqref{eq: orthogonality error relation}. Iteration yields
 	\begin{equation*}
 		\|O_n\|
 		\leqslant
 		\prod_{j=0}^{n-1}
 		\left(1+\frac{\tau_j}{E(U_0)}\right)^{-1}
 		\|O_0\|.
 	\end{equation*}
 	Since $\tau_j<\lambda_1$ and $t_n=\sum_{j=0}^{n-1}\tau_j\to\infty$,
 	\begin{equation*}
 		\|O_n\|
 		\leqslant
 		\exp\left(-\frac{t_n}{E(U_0)+\lambda_1}\right)\|O_0\|
 		\longrightarrow0.
 	\end{equation*}
 \end{proof}

 \subsubsection{Convergence}
 
 Combining energy decay with asymptotic orthogonality, we next prove
 convergence of the discrete iterates.
 \begin{theorem}\label{thm: convergence of Un}
 	Let $U_0\in\mathcal{S}$. Assume that there is a constant $\tau_{\min}>0$ such that
 	\begin{equation*}
 		0<\tau_{\min}\leqslant\tau_n
 		\leqslant
 		\min\left\{
 		\frac{\lambda_1}{2},
 		\frac{1}{4c_a},
 		\frac{\lambda_1}{4c_e\lambda_{\max}(\langle U_0,U_0\rangle)}
 		\right\},
 		\qquad n\geqslant0.
 	\end{equation*}
 	Then
 	\begin{equation*}
 		\lim_{n\to\infty}\|[U_n]-[V_*]\|_a=0.
 	\end{equation*}
 	The energy and gradient satisfy
 	\begin{equation*}
 		\begin{aligned}
 		&	\lim_{n\to\infty}E(U_n)=E(V_*),\\
 		&	\lim_{n\to\infty}\|\nabla_GE_{\mathcal G}(U_n)\|_a=0.
 		\end{aligned}
 	\end{equation*}
 \end{theorem}

 \begin{proof}
 	The time-step restriction and Theorem \ref{thm: Energy E(Un) decay} give
 	\begin{equation}\label{eq: discrete energy decrease for convergence}
 		E(U_n)-E(U_{n+1})
 		\geqslant
 		\frac{\lambda_1\tau_{\min}}
 		{4\lambda_{\max}(\langle U_0,U_0\rangle)}
 		\|\mathcal{A}_{U_n}U_n\|_a^2.
 	\end{equation}
 	Since the energy is bounded below,
 	\begin{equation*}
 		\sum_{n=0}^\infty\|\mathcal{A}_{U_n}U_n\|_a^2<\infty,
 		\qquad
 		\|\mathcal{A}_{U_n}U_n\|_a\longrightarrow0.
 	\end{equation*}

 	Energy decay and Theorem \ref{thm: stability} imply $U_n\in\mathcal S$ for all $n$. Suppose that the Grassmann distance does not converge to zero. Then there exist $\delta>0$ and a subsequence $\{U_{n_j}\}_{j=0}^\infty$ such that
 	\begin{equation*}
 		\|[U_{n_j}]-[V_*]\|_a\geqslant\delta,
 		\qquad j\geqslant0.
 	\end{equation*}
 	Boundedness in $\big(H_0^1(\Omega)\big)^N$ gives a further subsequence, not relabeled, and a limit $\bar U$ such that
 	\begin{equation*}
 		U_{n_j}\rightharpoonup\bar U
 		\quad\text{weakly in }\big(H_0^1(\Omega)\big)^N,
 		\qquad
 		U_{n_j}\to\bar U
 		\quad\text{strongly in }\big(L^2(\Omega)\big)^N.
 	\end{equation*}
 	Theorem \ref{thm: asymptotic orthogonality of Un} yields
 	\begin{equation*}
 		\langle\bar U,\bar U\rangle=I_N.
 	\end{equation*}
 	Furthermore,
 	\begin{equation*}
 		A_{n_j}\to\langle\bar U,\mathcal G\bar U\rangle,
 		\qquad
 		A_{n_j}^{-1}\to
 		\langle\bar U,\mathcal G\bar U\rangle^{-1}.
 	\end{equation*}
 	The identity
 	\begin{equation*}
 		U_{n_j}
 		=\mathcal GU_{n_j}\langle U_{n_j},U_{n_j}\rangle A_{n_j}^{-1}
 		-\mathcal A_{U_{n_j}}U_{n_j}A_{n_j}^{-1}
 	\end{equation*}
 	then shows that $U_{n_j}\to\bar U$ strongly in $\big(H_0^1(\Omega)\big)^N$. Indeed, $\mathcal G$ maps strong $L^2$ convergence to strong $H_0^1$ convergence, while $A_{n_j}^{-1}$ is uniformly bounded by Lemma \ref{lem:bound of UGU}.

 	Passing to the limit in the same identity gives
 	\begin{equation*}
 		\mathcal G\bar U
 		-\bar U\langle\bar U,\mathcal G\bar U\rangle=0.
 	\end{equation*}
 	Thus, $\bar U\in\mathcal S$ is a critical point. The uniqueness of the critical class in $\mathcal S$ implies $[\bar U]=[V_*]$. This contradicts the lower bound on $\|[U_{n_j}]-[V_*]\|_a$ because $U_{n_j}\to\bar U$ strongly. Therefore,
 	\begin{equation*}
 		\|[U_n]-[V_*]\|_a\longrightarrow0.
 	\end{equation*}
 	Orthogonal invariance and continuity of the energy yield
 	\begin{equation*}
 		E(U_n)\longrightarrow E(V_*).
 	\end{equation*}
	Finally, \eqref{eq: vector field splitting} gives
	$\|\nabla_GE_{\mathcal G}(U_n)\|_a\to0$,
	since $\|\mathcal A_{U_n}U_n\|_a\to0$ and
	$\|\mathcal GU_nO_n\|_a\to0$.
 \end{proof}
 
\subsubsection{Exponential convergence}

Using the convergence of the discrete iterates, we establish the
following estimates to derive exponential convergence rates.
\begin{lemma}\label{lem:discrete local estimates}
	Under the hypotheses of Theorem \ref{thm: convergence of Un}, there
	exist $n_0\in\mathbb N$ and $C>0$ such that, for all $n\geqslant n_0$,
	\begin{equation*}
		\|O_n\|\leqslant
		C\left(1+\frac{\tau_{\min}}{E(U_0)}\right)^{-n},
	\end{equation*}
	and
	\begin{equation*}
		\begin{aligned}
			E(U_n)-E(V_*)
			&\leqslant C\|\mathcal A_{U_n}U_n\|_a^2+C\|O_n\|,\\
			\|O_n\|&\leqslant C\big(E(U_n)-E(V_*)\big),\\
			\|[U_n]-[V_*]\|_a^2+\|\mathcal A_{U_n}U_n\|_a^2
			&\leqslant C\big(E(U_n)-E(V_*)\big).
		\end{aligned}
	\end{equation*}
\end{lemma}
\begin{proof}
	The proof is given in Appendix \ref{app:lemma-proofs}.
\end{proof}

Together with the preceding convergence and energy estimates, these bounds
give the following exponential rates.
\begin{theorem}\label{thm: exponential rate of Un}
	Assume the hypotheses of Theorem \ref{thm: convergence of Un}. Then there exist $n_0\in\mathbb{N}$, $C>0$, $\rho\in(0,1)$, and $Q_*\in\mathcal{O}^N$ such that
	\begin{equation*}
		\begin{aligned}
			\|\nabla_GE_{\mathcal G}(U_n)\|_a
			&\leqslant C\rho^n,\\
			\|U_n-V_*Q_*\|_a
			&\leqslant C\rho^n,\\
			0\leqslant E(U_n)-E(V_*)
			&\leqslant C\rho^{2n},
		\end{aligned}
		\qquad n\geqslant n_0.
	\end{equation*}
\end{theorem}

\begin{proof}
		Fix constants $C_1,C_2>0$ as in Lemma
		\ref{lem:discrete local estimates} such that, for all $n\geqslant n_0$,
		\begin{equation*}
			E(U_n)-E(V_*)
			\leqslant
			C_1\|\mathcal A_{U_n}U_n\|_a^2+C_2\|O_n\|.
		\end{equation*}
		Together with \eqref{eq: discrete energy decrease for convergence},
		this gives
		\begin{equation*}
			\begin{aligned}
				E(U_{n+1})-E(V_*)
				&\leqslant
				\left(
				1-\frac{\lambda_1\tau_{\min}}
				{4C_1\lambda_{\max}(\langle U_0,U_0\rangle)}
				\right)
				\big(E(U_n)-E(V_*)\big)\\
				&\quad+
				\frac{\lambda_1\tau_{\min}C_2}
				{4C_1\lambda_{\max}(\langle U_0,U_0\rangle)}
				\|O_n\|.
			\end{aligned}
		\end{equation*}
	Increasing $C_1$ if necessary, there exist $q_0\in(0,1)$ and $C_0>0$
	such that
	\begin{equation}\label{eq: discrete asymptotic residual contraction}
		E(U_{n+1})-E(V_*)
		\leqslant
		q_0\big(E(U_n)-E(V_*)\big)
		+C_0\|O_n\|,
		\qquad n\geqslant n_0.
	\end{equation}
	Choose
	\begin{equation*}
		\rho\in\left(
		\max\left\{
		\sqrt{q_0},
		\left(1+\frac{\tau_{\min}}{E(U_0)}\right)^{-1/2}
		\right\},
		1
		\right).
	\end{equation*}
	Then \eqref{eq: new On bound} also gives
	\begin{equation*}
		\|O_n\|\leqslant C\rho^{2n}.
	\end{equation*}
	Iterating \eqref{eq: discrete asymptotic residual contraction} gives
		\begin{equation*}
			\begin{aligned}
				E(U_n)-E(V_*)
				\leqslant
				q_0^{n-n_0}\big(E(U_{n_0})-E(V_*)\big)+
				C\sum_{j=n_0}^{n-1}
				q_0^{n-1-j}
				\left(1+\frac{\tau_{\min}}{E(U_0)}\right)^{-j}.
			\end{aligned}
		\end{equation*}
	By the choice of $\rho$,
		\begin{equation*}
			E(U_n)-E(V_*)\leqslant C\rho^{2n}.
	\end{equation*}
	Lemma \ref{lem:discrete local estimates} then gives
	\begin{equation*}
		\|\mathcal A_{U_n}U_n\|_a
		+\|[U_n]-[V_*]\|_a
		\leqslant C\rho^n.
	\end{equation*}

	Using \eqref{eq: vector field splitting}, we obtain
	\begin{equation*}
		\|\nabla_GE_{\mathcal G}(U_n)\|_a
		\leqslant
		\|\mathcal A_{U_n}U_n\|_a
		+\|\mathcal GU_nO_n\|_a
		\leqslant C\rho^n.
	\end{equation*}

	The time-step bound, the boundedness of $\{\mathcal GU_n\}$, and the
	estimate
	\begin{equation*}
		\|\mathcal A_{U_n}\tilde U_{n+\frac12}\|_a
		\leqslant
		\frac43\|\mathcal A_{U_n}U_n\|_a
	\end{equation*}
	from the proof of Theorem \ref{thm: Energy E(Un) decay} give
		\begin{equation*}
			\begin{aligned}
				\|U_{n+1}-U_n\|_a
			&\leqslant
			\tau_n\|\mathcal GU_{n+1}O_{n+1}\|_a
			+\tau_n\|\mathcal A_{U_n}\tilde U_{n+\frac12}\|_a\\
			&\leqslant
			C\|O_{n+1}\|
			+C\|\mathcal A_{U_n}U_n\|_a
				\leqslant C\rho^n.
			\end{aligned}
		\end{equation*}
	Hence $\{U_n\}$ is Cauchy in $\big(H_0^1(\Omega)\big)^N$ and converges
	strongly to some $\bar U$. Theorem \ref{thm: convergence of Un} gives
	$[\bar U]=[V_*]$. Thus $\bar U=V_*Q_*$ for some
	$Q_*\in\mathcal{O}^N$. Summing the increments gives
	\begin{equation*}
		\|U_n-V_*Q_*\|_a
		\leqslant
		\sum_{j=n}^{\infty}\|U_{j+1}-U_j\|_a
		\leqslant C\rho^n.
	\end{equation*}
\end{proof}
\section{Numerical experiments}\label{section:Numerical experiments}

We now examine the numerical behavior of the two-stage method,
beginning with its implementation and error measures.
We then present computations for a two-dimensional harmonic
oscillator and the hydrogen atom, followed by numerical
observations on the time step, the spectral gap, and the cost of applying the
inverse operator.
The numerical results are consistent with the theoretical analysis.

\subsection{Implementation of the iteration}

Algorithm \ref{alg:Discretization scheme} summarizes the implementation.
In this section, applications of $\mathcal G$ are evaluated by solving
the corresponding finite element source problems.
We compute the first stage through four substeps and use one Picard update
for the second stage. Define
\begin{equation*}
	C_n=\langle\mathcal GU_n,\tilde U_{n+\frac12}\rangle,
	\qquad
	D_n=\langle U_n,\tilde U_{n+\frac12}\rangle.
\end{equation*}
Taking inner product matrices of the midpoint relation with
$\mathcal GU_n$ and $U_n$ gives the coupled linear system in step (iii).

\begin{algorithm}[!ht]
	\caption{Quasi-orthogonal iteration}
	\label{alg:Discretization scheme}
	\begin{algorithmic}[1]
		\STATE Given $U_0\in(V^{N_g})^N$, time steps $\{\tau_n\}$, and tolerances $\epsilon,
		\epsilon'>0$, set $n=0$;
		\WHILE{$\|\nabla_GE_{\mathcal G}(U_n)\|_a>\epsilon \lor
		\|I_N-\langle U_n,U_n\rangle\|>\epsilon'$}
		\STATE \textbf{First stage}
		\STATE \textit{(i)} Solve the $N$ independent source problems
		\begin{equation}
			\label{eq:inverse-operator equation}
			(-\Delta+\mathcal V)\mathcal GU_n=U_n.
		\end{equation}
		\STATE \textit{(ii)} Compute $\langle U_n,U_n\rangle$ and
		\begin{equation*}
			A_n=\langle U_n,\mathcal GU_n\rangle,
			\qquad
			B_n=\langle\mathcal GU_n,\mathcal GU_n\rangle.
		\end{equation*}
		\STATE \textit{(iii)} Solve for $C_n$ and $D_n$:
		\begin{equation*}
			\begin{pmatrix}
				I_N+\frac{\tau_n}{2}A_n & -\frac{\tau_n}{2}B_n\\[2pt]
				\frac{\tau_n}{2}\langle U_n,U_n\rangle & I_N-\frac{\tau_n}{2}A_n
			\end{pmatrix}
			\begin{pmatrix}C_n\\D_n\end{pmatrix}
			=
			\begin{pmatrix}A_n\\\langle U_n,U_n\rangle\end{pmatrix}.
		\end{equation*}
		\STATE \textit{(iv)} Set
		\begin{equation*}
			\hat U_{n+1}=U_n-\tau_nU_nC_n+\tau_n\mathcal GU_nD_n.
		\end{equation*}
		\STATE \textbf{Second stage}
		\STATE Solve the source problems for $\mathcal G\hat U_{n+1}$.
		\STATE Apply one Picard update:
		\begin{equation}\label{equ: one Picard correction}
			U_{n+1}=\hat U_{n+1}-\tau_n\mathcal G\hat U_{n+1}
			\left(\langle U_n,U_n\rangle-I_N\right).
		\end{equation}
		\STATE Set $n=n+1$;
		\ENDWHILE
	\end{algorithmic}
\end{algorithm}

The first stage preserves the Gram matrix:
$\langle\hat U_{n+1},\hat U_{n+1}\rangle=\langle U_n,U_n\rangle$.
Hence, the correction uses the Gram matrix of the current iterate.

\subsection{Numerical settings and error measures}

The computations were performed on the LSSC-IV platform at the Academy of
Mathematics and Systems Science, Chinese Academy of Sciences. All computations
use quadratic finite elements \cite{sun2026finite}. We denote the finite element space by
$V^{N_g}\subset H_0^1(\Omega)$, where $N_g$ is the number of degrees of
freedom. We use randomly generated nonorthogonal initial data.\footnote{
The analysis imposes additional conditions on the initial data.
In these computations, we use random nonorthogonal initial data
without enforcing those conditions.}
A uniform time step is used unless stated otherwise. Within each comparison, the initial-data construction
is kept fixed. The iteration stops when
\begin{equation*}
	\|\nabla_GE_{\mathcal G}(U_n)\|_a\leqslant10^{-5}
	\quad\text{and}\quad
	\|I_N-\langle U_n,U_n\rangle\|\leqslant10^{-12}.
\end{equation*}

At termination, we diagonalize
$\langle U_{\mathrm{end}},U_{\mathrm{end}}\rangle_a$ to obtain individual
eigenpair approximations. Let $Q\in\mathcal O^N$ satisfy
\begin{equation*}
	Q^{\top}\langle U_{\mathrm{end}},U_{\mathrm{end}}\rangle_aQ
	=\Lambda_{\mathrm{end}},
	\qquad
	\Lambda_{\mathrm{end}}
	=\operatorname{diag}(\lambda_{1,h},\ldots,\lambda_{N,h}).
\end{equation*}
Then the columns of $V_{\mathrm{end}}=U_{\mathrm{end}}Q$ are the corresponding
eigenfunction approximations.

We monitor the energy, the orthogonality error
$\|I_N-\langle U_n,U_n\rangle\|$, and the gradient norm
$\|\nabla_GE_{\mathcal G}(U_n)\|_a$. Theorem
\ref{thm: exponential rate of Un} shows that the iterates converge to a fixed
orthogonal transformation of the target eigenfunctions.
We use $U_{\mathrm{end}}$ to approximate this limit and define the relative
iterate error by
\begin{equation}\label{eq: relative iterate error}
	\operatorname{err}_{U_n}
	=\frac{\|U_n-U_{\mathrm{end}}\|_a}{\|U_{\mathrm{end}}\|_a}.
\end{equation}

\subsection{Large-scale computations}

We compute 496 eigenpairs of the two-dimensional harmonic oscillator and
30 eigenpairs of the hydrogen atom.
Both problems are solved on truncated domains with homogeneous
Dirichlet boundary conditions.
We use $\tau=1$ and solve the source problems \eqref{eq:inverse-operator equation} by the
conjugate gradient method with relative tolerance $10^{-12}$ and absolute
tolerance $10^{-13}$.

\label{example:harmonic}
	The eigenvalue problem for the two-dimensional harmonic oscillator \cite{ReedSimonIV} is given by
	\begin{equation*}
		-\frac{1}{2} \Delta u+\frac{1}{2}|x|^2 u = \lambda u, \qquad \int_{\mathbb{R}^2} u^2 = 1,
	\end{equation*}
	where $|x|=\sqrt{|x_1|^2+|x_2|^2}$. Its eigenvalues are
	$
	\lambda_{n_1, n_2} = \left(n_1+\frac{1}{2}\right)+\left(n_2+\frac{1}{2}\right),\  n_1, n_2=0,1, \cdots,
	$
	with eigenfunctions given, up to normalization, by
	$
	u_{n_1, n_2}(x) = \mathcal{H}_{n_1}(x_1) e^{-x_1^2 / 2} \mathcal{H}_{n_2}(x_2) e^{-x_2^2 / 2},\  n_1, n_2=0,1, \cdots,
	$
	where $\mathcal{H}_n$ denotes the $n$-th Hermite polynomial.
	
	For the harmonic oscillator, we use $\Omega=(-15,15)^2$
	and compute the first $N=496$ eigenpairs with $N_g=229441$
	degrees of freedom.
	The iteration terminates after 8291 steps.
	Figure \ref{fig:harmonic-N496-convergence} shows how the energy,
	gradient norm, orthogonality error, and relative iterate error
	evolve during the iteration.
% Figure data: job 748319; relative iterate error: omit the final five samples.
\begin{figure}[H]
	\centering
	\includegraphics[width=0.98\linewidth]{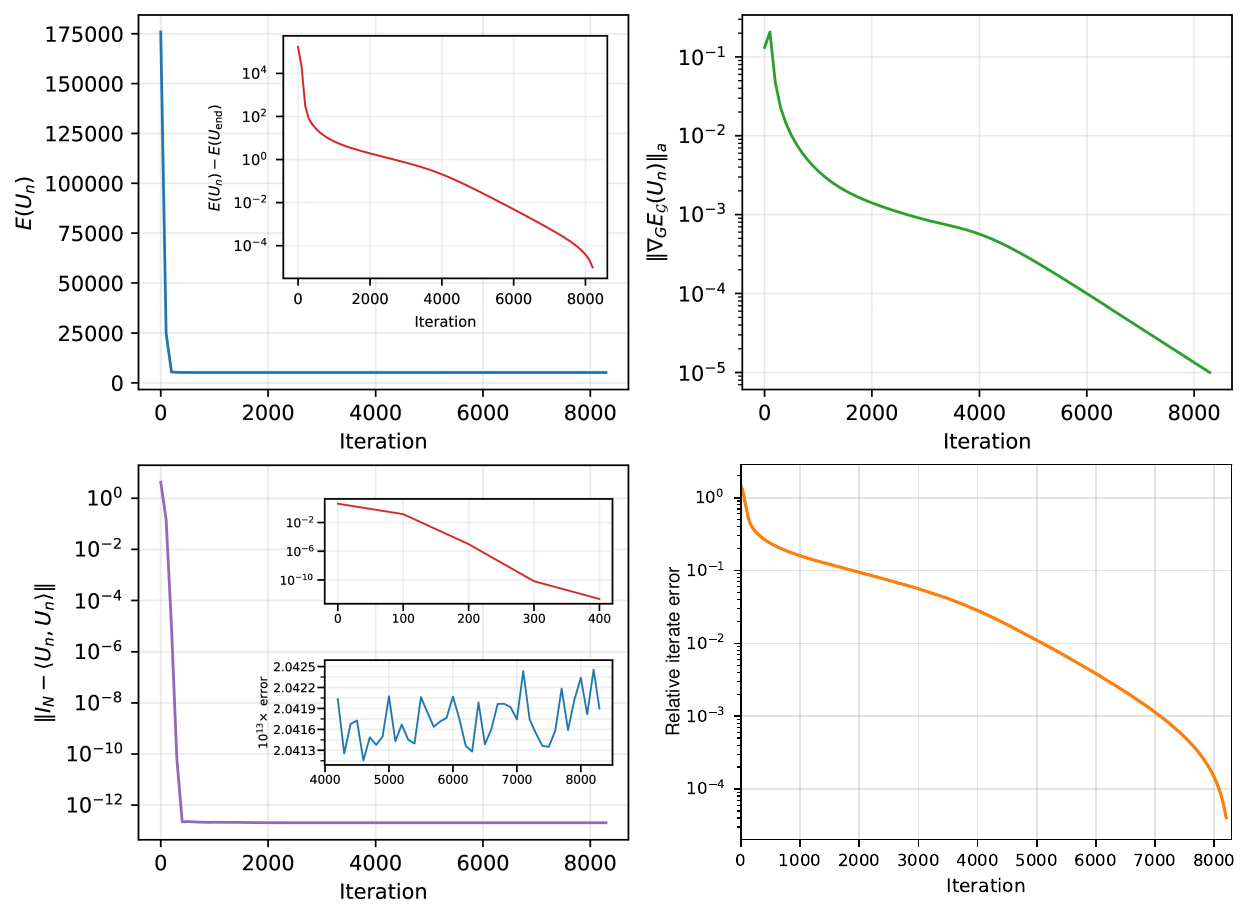}
	\caption{Convergence histories for the harmonic oscillator with $N=496$: (a) energy, (b) gradient norm, (c) orthogonality error, and (d) relative iterate error}
	\label{fig:harmonic-N496-convergence}
\end{figure}

%	% Job 748319: harmonic oscillator, N=496, 64 MPI ranks on 4 nodes.
%	\begin{table}[htbp]
%		\centering
%		\small
%		\caption{Numerical results for the harmonic oscillator with $N=496$}
%		\label{table:harmonic-N496-summary}
%		\begin{tabular}{rrrrrr}
%			\toprule
%			$N_g$ & $\tau$ & Iterations & $E(U_{\mathrm{end}})$
%			& Orthogonality error & Gradient norm \\
%			\midrule
%			229441 & 1.0 & 8291 & 5208.639525
%			& $2.042\times10^{-13}$ & $9.996\times10^{-6}$ \\
%			\bottomrule
%		\end{tabular}
%	\end{table}

\label{example:hydrogen}
	We consider the Schr{\"o}dinger eigenvalue problem for the hydrogen atom
	in three dimensions \cite{greiner2011quantum}:
	\begin{equation*}
		\left(-\frac{1}{2} \Delta-\frac{1}{|x|}\right) u = \lambda u, \qquad \int_{\mathbb{R}^3}|u|^2  = 1.
	\end{equation*}
	The eigenvalue $\lambda_n=-1/(2n^2)$ has multiplicity $n^2$.
	
	For the hydrogen atom, we use $\Omega=(-20,20)^3$
	and compute the first $N=30$ eigenpairs.
	The adaptive quadratic finite element mesh resolves the singularity
	at the origin \cite{dai2015convergence} and has
	$N_g=1{,}867{,}052$ degrees of freedom.
	The iteration terminates after 3989 steps.
	Figure \ref{fig:hydrogen-N30-convergence} shows the same convergence
	diagnostics for the hydrogen atom.
% Figure data: job 748309; relative iterate error: omit the final five samples.
\begin{figure}[H]
	\centering
	\includegraphics[width=0.98\linewidth]{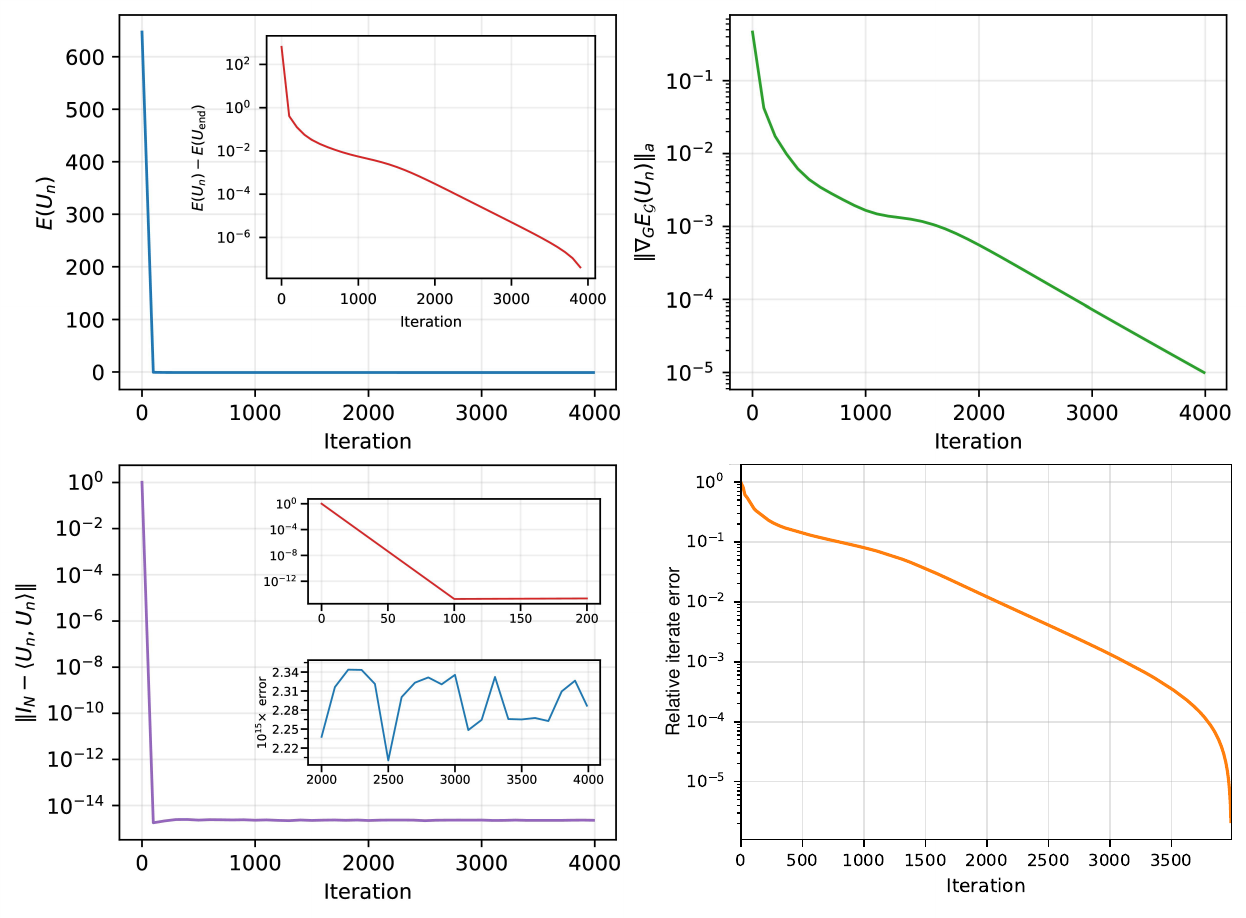}
	\caption{Convergence histories for the hydrogen atom with $N=30$: (a) energy, (b) gradient norm, (c) orthogonality error, and (d) relative iterate error}
	\label{fig:hydrogen-N30-convergence}
\end{figure}
%	% Job 748309: hydrogen atom, N=30, 32 MPI ranks on 4 nodes.
%	\begin{table}[htbp]
%		\centering
%		\small
%		\caption{Numerical results for the hydrogen atom with $N=30$}
%		\label{table:hydrogen-N30-summary}
%		\begin{tabular}{rrrrrr}
%			\toprule
%			$N_g$ & $\tau$ & Iterations & $E(U_{\mathrm{end}})$
%			& Orthogonality error & Gradient norm \\
%			\midrule
%			1867052 & 1.0 & 3989 & $-1.034785648$
%			& $2.287\times10^{-15}$ & $9.984\times10^{-6}$ \\
%			\bottomrule
%		\end{tabular}
%	\end{table}

	In both computations, the energy decreases, while the gradient norm and
	relative iterate error decay. The energy insets show exponential
	decay of the difference from the final computed energy $E(U_{\mathrm{end}})$.
	The orthogonality-error insets show an initial exponential decay
	followed by a plateau near $10^{-13}$ for the harmonic oscillator and
	$10^{-15}$ for the hydrogen atom.
	The observed energy decay, approach to orthogonality, and convergence
	are consistent with Theorems \ref{thm: Energy E(Un) decay},
	\ref{thm: asymptotic orthogonality of Un}, and
	\ref{thm: exponential rate of Un}, respectively.

\subsection{Further numerical observations}

We next use the two-dimensional harmonic oscillator on
$\Omega=(-5.5,5.5)^2$ with quadratic finite elements. Unless varied below,
we set $N=15$, $m=100$, $N_g=39601$, and $\tau=1$. We examine the dependence
of the iteration on the spatial mesh and the spectral gap, and then consider a
strategy for reducing the cost of the source problems.

\medskip
\noindent\textbf{(i) Mesh-independent time steps.}

We use four uniform meshes with $m=100,200,400,800$ cells in each coordinate
direction. The grid spacing is $h=11/m$.
For each mesh, we use the same initial-data construction and
test $\tau=0.5,1.0,1.5$.

% Jobs 739544, 739546, 739550, and 740151: mesh--time-step experiment.
\begin{table}[htbp]
	\centering
	\small
	\caption{Iteration counts for different spatial meshes and time steps}
	\label{table:mesh-timestep}
	\begin{tabular}{rrrrrr}
		\toprule
		$m$ & $h$ & $N_g$ & $\tau=0.5$ & $\tau=1.0$ & $\tau=1.5$ \\
		\midrule
		100 & 0.11000 & 39601   & 671 & 339 & 228 \\
		200 & 0.05500 & 159201  & 671 & 339 & 228 \\
		400 & 0.02750 & 638401  & 671 & 339 & 228 \\
		800 & 0.01375 & 2556801 & 671 & 339 & 228 \\
		\bottomrule
	\end{tabular}
\end{table}

For each of the three tested time steps, the iteration meets the stopping
criteria in the same number of steps on all four meshes
(Table \ref{table:mesh-timestep}).

\medskip
\noindent\textbf{(ii) Dependence on the spectral gap of the inverse operator.}

We use the spectral gap of the discrete inverse operator,
\begin{equation*}
	\delta_{N,h}=\frac{1}{\lambda_{N,h}}
	-\frac{1}{\lambda_{N+1,h}},
\end{equation*}
where the reference eigenvalues are computed in the same finite element
space. The mesh, time step, and initial-data construction are fixed while $N$
is varied.

% Jobs 751039_1--751039_7: gap experiment; job 751382: discrete reference spectrum.
% Figure data: jobs 751039_1--751039_7 and 751382.
\begin{figure}[htbp]
	\centering
	\begin{minipage}[t]{0.4\textwidth}
		\vspace{7pt}
		\centering
		\small
		\setlength{\tabcolsep}{4pt}
		\renewcommand{\arraystretch}{1.4}
		\begin{tabular}{rrr}
			\toprule
			$N$ & $\delta_{N,h}$ & Iterations \\
			\midrule
			15  & $3.333280\times10^{-2}$ & 308 \\
			28  & $1.785598\times10^{-2}$ & 500 \\
			45  & $1.110413\times10^{-2}$ & 865 \\
			66  & $7.510947\times10^{-3}$ & 1215 \\
			91  & $5.121280\times10^{-3}$ & 1728 \\
			120 & $2.943564\times10^{-3}$ & 3530 \\
			153 & $6.799476\times10^{-4}$ & 16979 \\
			\bottomrule
		\end{tabular}
			\vspace{7pt}
		\captionof{table}{Spectral gaps of the inverse operator and iteration counts}
		\label{table:inverse-gap}
	\end{minipage}
	\hfill
	\begin{minipage}[t]{0.55\textwidth}
		\vspace{0pt}
		\centering
		\includegraphics[width=\linewidth]{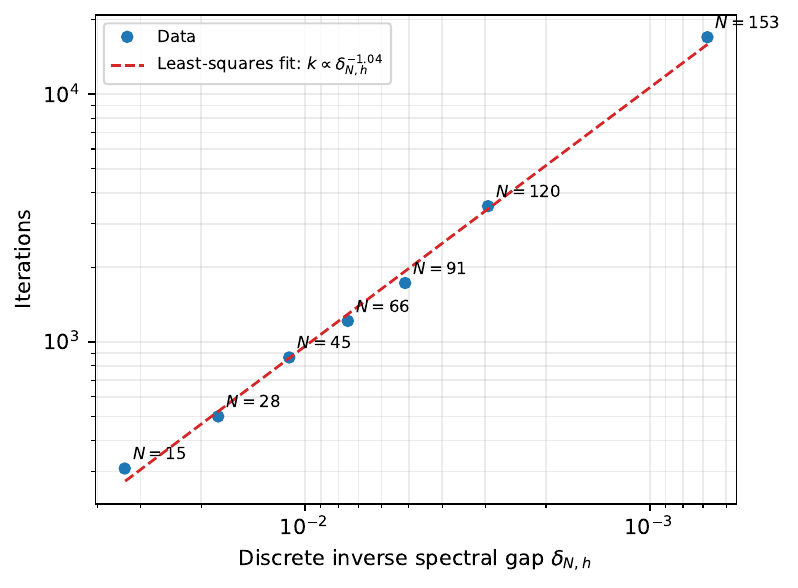}
		\caption{Iteration counts versus the spectral gap of the inverse operator}
		\label{fig:inverse-gap}
	\end{minipage}
\end{figure}

Table \ref{table:inverse-gap} and Figure \ref{fig:inverse-gap} show that the
iteration count increases as $\delta_{N,h}$ decreases, consistent with the
dependence of the convergence rate on the spectral gap. A least-squares fit of
the form $\mathrm{Iterations}=C\delta_{N,h}^{\,p}$ gives $p=-1.04$.

\medskip
\noindent\textbf{(iii) Reducing the cost of source solves.}

The source solves account for most of the cost of each outer iteration. We use
the conjugate gradient method and take the solution from the preceding outer
iteration as the initial guess. We examine whether using a larger
source-solve tolerance early in the iteration reduces the total iteration time.

The total iteration time is the wall-clock time of all outer iterations,
excluding construction of the discrete problem. The times are averaged over
two runs, and the CG iteration counts are summed over all source solves in
both stages.
For a source problem $A_hx=b$, let $r^{(j)}=b-A_hx^{(j)}$ be the residual at
the $j$th CG iteration. The tolerance $\varepsilon$ listed below is used as
both the relative and absolute residual tolerance. The CG iteration stops when
\begin{equation*}
	\|r^{(j)}\|\leqslant
	\max\left\{\varepsilon\|r^{(0)}\|,\varepsilon\right\}.
\end{equation*}
The initial source-solve tolerance is the value of $\varepsilon$ used in
the CG solves for applying $\mathcal G$ before the switch. In the dynamic
strategy, both CG tolerances are set to $10^{-12}$ when the outer
gradient norm $\|\nabla_GE_{\mathcal G}(U_n)\|_a$ falls below the threshold
$\gamma$.

\medskip
\noindent\textit{Effect of delaying the switch.}
We fix the initial source-solve tolerance at $10^{-8}$ and vary $\gamma$.
Table \ref{table:dynamic-threshold} gives the switch step,
outer and CG iteration counts, and total iteration time.

% Job 752469: preliminary fixed/dynamic comparison with gamma=1e-4.
% Jobs 752470 and 752519: two runs for every switching threshold.
\begin{table}[htbp]
	\centering
	\small
	\caption{Effect of the gradient-norm threshold with initial source-solve tolerance $10^{-8}$}
	\label{table:dynamic-threshold}
	\begin{tabular}{lrrrr}
		\toprule
		\shortstack{Gradient-norm\\threshold $\gamma$} & Switch step & Outer iterations & CG iterations
		& \shortstack{Total iteration time (s)} \\
		\midrule
		$9\times10^{-4}$       & 186 & 319 & 61143 &  96.17 \\
		$3\times10^{-4}$       & 218 & 319 & 57273 &  92.78 \\
		$1\times10^{-4}$       & 251 & 319 & 53259 &  88.57 \\
		$5\times10^{-5}$       & 271 & 319 & 50829 &  86.26 \\
		$2\times10^{-5}$       & 298 & 319 & 47548 &  83.54 \\
		\bottomrule
	\end{tabular}
\end{table}

For the tested thresholds, delaying the switch reduces the total number
of CG iterations. All runs meet the stopping criteria in 319 outer
iterations.

\medskip
\noindent\textit{Reducing iteration time with dynamic tolerances.}
We fix the gradient-norm threshold at $\gamma=9\times10^{-5}$ and compare
three initial source-solve tolerances.

% Job 752520: two runs for every initial tolerance.
\begin{table}[htbp]
	\centering
	\small
	\caption{Effect of the initial source-solve tolerance with $\gamma=9\times10^{-5}$}
	\label{table:dynamic-initial-tolerance}
	\setlength{\tabcolsep}{5pt}
	\begin{tabular}{rrrrr}
		\toprule
		\shortstack{Initial source-solve\\tolerance} & \shortstack{Switch step}
		& \shortstack{Outer iterations} & \shortstack{CG iterations}
		& \shortstack{Total iteration\\time (s)} \\
		\midrule
		$10^{-8}$ & 254 & 319 & 52891 & 86.83 \\
		$10^{-6}$ & 254 & 319 & 35572 & 71.47 \\
		$10^{-5}$ & 254 & 319 & 27857 & 62.53 \\
		\bottomrule
	\end{tabular}
\end{table}

In Table \ref{table:dynamic-initial-tolerance}, all three initial tolerances
give the same number of outer iterations, and $10^{-5}$ gives the shortest
total iteration time. We then use $10^{-5}$ as the initial tolerance and
test two switching thresholds. Table \ref{table:inverse-cost} compares
the resulting strategies with a fixed source-solve tolerance of $10^{-12}$.

% Job 756010: fixed accuracy and gamma=2e-5, each repeated twice.
% Job 756027: fixed accuracy and gamma=5e-5, each repeated twice.
\begin{table}[htbp]
	\centering
	\small
	\caption{Comparison of dynamic source-solve strategies}
	\label{table:inverse-cost}
	\setlength{\tabcolsep}{4pt}
	\begin{tabular}{lrrrrr}
		\toprule
		Strategy & \shortstack{Gradient-norm\\threshold $\gamma$} & \shortstack{Switch step}
		& \shortstack{Outer iterations} & \shortstack{CG iterations}
		& \shortstack{Total iteration\\time (s)} \\
		\midrule
		Fixed accuracy & --                 & --  & 319 & 89652 & 125.47 \\
		Dynamic        & $5\times10^{-5}$ & 269 & 320 & 25173 &  60.67 \\
		Dynamic        & $2\times10^{-5}$ & 293 & 324 & 21382 &  57.92 \\
		\bottomrule
	\end{tabular}
\end{table}

Both dynamic strategies meet the stopping criteria. Relative to the
fixed-tolerance run, the total iteration time decreases by $51.6\%$ for
$\gamma=5\times10^{-5}$ and by $53.8\%$ for $\gamma=2\times10^{-5}$.
For the latter threshold, the switch occurs at step 293 of 324.

	\section{Conclusion}\label{sec: conclusion}
	Computing many Schr{\"o}dinger eigenpairs requires mutually orthogonal
	eigenfunction approximations. We developed a quasi-orthogonal evolution
	model based on the inverse operator and a two-stage time discretization.
	The continuous trajectory and discrete iterates approach orthogonality
	from nonorthogonal initial data without any orthogonalization.
	For the continuous evolution equation and its two-stage time discretization,
	we established well-posedness,
	energy decay, and exponential convergence. Numerical experiments illustrate
	the convergence behavior and confirm that the time steps remain
	effective under mesh refinement.
	Future work will focus on developing accelerated algorithms and
	extending the quasi-orthogonal framework to nonlinear eigenvalue problems.

\appendix
\section{Proofs of auxiliary lemmas}\label{app:lemma-proofs}
This appendix collects the proofs of the auxiliary lemmas used in the main text.

\begin{proof}[Proof of Lemma \ref{lem: Local Lipschitz property}]
		Using the formulation $\nabla_G E_{\mathcal{G}}(U) = \mathcal{G}U - U \langle \mathcal{G}U, U \rangle$, the triangle inequality yields
		\begin{equation*}
			\begin{aligned}
				\left\| \nabla_G E_{\mathcal{G}}(U) - \nabla_G E_{\mathcal{G}}(V) \right\|_a 
				\leqslant & \| \mathcal{G}(U - V) \|_a + \| U \langle \mathcal{G}U, U \rangle - V \langle \mathcal{G}V, V \rangle \|_a \\
				\leqslant & \left\| \mathcal{G}\left(U-V\right)\right\|_a + \left\|(U-V)\left\langle  \mathcal{G}U, U \right\rangle\right\|_a \\
				& + \left\| V \left(\left\langle \mathcal{G}U, U \right\rangle   - \left\langle \mathcal{G}V, V \right\rangle  \right)\right\|_a.
			\end{aligned}
		\end{equation*}
		The boundedness of $\mathcal G$ gives
		\begin{equation*}
			\begin{aligned}
					\|\mathcal{G}(U - V)\|_a& = \left( \mathcal{G}(U - V), \mathcal{G}(U - V)\right)_a^{\frac{1}{2}} =  \left( U - V, \mathcal{G}(U - V)\right)^{\frac{1}{2}}  
					\\ &\leqslant \frac{1}{\sqrt{\lambda_1}} \|U - V\| \leqslant \frac{c_{\Omega}}{\sqrt{\lambda_1}} \|U - V\|_a.
			\end{aligned}
		\end{equation*}
		For the second term, the trace identity and the Poincar\'e inequality give
		\begin{equation*}
			\begin{aligned}
				\left\|(U-V)\left\langle  \mathcal{G}U, U \right\rangle\right\|_a  &= \operatorname{tr}\left(\left\langle U-V, U-V \right\rangle_a \cdot \left\langle U,\mathcal{G} U \right\rangle^2 \right)^{\frac{1}{2}} 
				\\ &\leqslant  \frac{1}{\lambda_1}\left\|U\right\|^2 \cdot \left\| U-V \right\|_a 
				\leqslant \frac{c_{\Omega}^2}{\lambda_1}\left\|U\right\|_a^2 \cdot \left\| U-V \right\|_a. 
			\end{aligned}
		\end{equation*}
		For the third term, decompose the inner product difference as
		\begin{equation*}
			\begin{aligned}
				\left\| V \left(\left\langle \mathcal{G}U, U \right\rangle   - \left\langle \mathcal{G}V, V \right\rangle  \right)\right\|_a &=  \left\| V \left(\left\langle \mathcal{G}(U-V), U \right\rangle   + \left\langle \mathcal{G}V, U-V \right\rangle  \right)\right\|_a  
				\\ &\leqslant \frac{1}{\lambda_1}   \left\| V \right\|_a  \cdot \left(  \left\|U-V\right\|\|U\| +\|V\| \left\| U-V\right\| \right)
				\\ &\leqslant \frac{2c_{\Omega}^2 M^2}{\lambda_1}\left\|U-V\right\|_a.
			\end{aligned}
		\end{equation*}
		Taking $C_M \coloneqq \frac{c_{\Omega}}{\sqrt{\lambda_1}}
		+\frac{3c_{\Omega}^2M^2}{\lambda_1}$ gives
		\begin{equation*}
			\begin{aligned}
				\left\| \nabla_G E_{\mathcal{G}}(U)- \nabla_G E_{\mathcal{G}}(V)  \right\|_a \leqslant C_M \left\| U-V\right\|_a.
			\end{aligned}
		\end{equation*}
	\end{proof}

\begin{proof}[Proof of Lemma \ref{lem:local existence}]
	Lemma \ref{lem: Local Lipschitz property} shows that the vector field in
	\eqref{equ: evolution problem} is locally Lipschitz on
	$\big(H_0^1(\Omega)\big)^N$. The Picard--Lindel\"of theorem therefore
	gives a unique local solution, which extends uniquely to a maximal interval
	$[0,T_{\max})$.
\end{proof}

\begin{proof}[Proof of Lemma \ref{lem:bound of UGU}]
 	Let $U \in \mathcal{M}_{\geqslant}^{N}$ and let $x \in \mathbb{R}^N$ satisfy $x^\top x = 1$. The upper bound follows from the boundedness of $\mathcal{G}$:
 	\begin{equation*}
 		x^\top\langle U,\mathcal{G}U\rangle x = (Ux,\mathcal{G}Ux) \leqslant \frac{1}{\lambda_1}\|Ux\|^2 \leqslant \frac{\lambda_{\max}(\langle U,U\rangle)}{\lambda_1}.
 	\end{equation*}
 	For the lower bound, the Cauchy--Schwarz inequality and the energy definition yield
 	\begin{equation*}
 		\begin{aligned}
 			1 = x^\top x &\leqslant \|Ux\|^2 = (Ux, Ux) = (Ux, \mathcal{G}Ux)_a 
 			\\
 			&\leqslant \|Ux\|_a \|\mathcal{G}Ux\|_a 
 			\leqslant \|U\|_a \|\mathcal{G}Ux\|_a 
 			\\
 			&= \sqrt{2E(U)} \sqrt{(\mathcal{G}Ux, \mathcal{G}Ux)_a} 
 			\\
 			&= \sqrt{2E(U)} \sqrt{(Ux, \mathcal{G}Ux)}.
 		\end{aligned}
 	\end{equation*}
 	Squaring both sides implies $(Ux, \mathcal{G}Ux) \geqslant \frac{1}{2E(U)}$. Consequently, the smallest eigenvalue satisfies
 	\begin{equation*}
 		\lambda_{\min}(\langle U, \mathcal{G}U \rangle) = \min_{x \in \mathbb{R}^N, \|x\|=1} (Ux, \mathcal{G}Ux) \geqslant \frac{1}{2E(U)}.
 	\end{equation*}
 \end{proof}

\begin{proof}[Proof of Lemma \ref{lem:continuous asymptotic estimates}]
	For each $t\geqslant0$, compactness of $\mathcal O^N$ gives
	$Q(t)\in\mathcal O^N$ such that
	\begin{equation*}
		\|U(t)-V_*Q(t)\|_a=\|[U(t)]-[V_*]\|_a.
	\end{equation*}
	Theorem \ref{thm: convergence} gives
	\begin{equation*}
		\|U(t)-V_*Q(t)\|_a\longrightarrow0,
	\end{equation*}
	while Theorem \ref{thm: asymptotic orthogonality} gives
	\begin{equation*}
		\left\|\langle U(t),U(t)\rangle-I_N\right\|
		\longrightarrow0.
	\end{equation*}
	Since $\langle V_*,V_*\rangle=I_N$ and
	$\mathcal GV_*=V_*\Lambda^{-1}$,
	\begin{equation*}
		\begin{aligned}
			&\langle V_*,U(t)\rangle-Q(t)
			=\langle V_*,U(t)-V_*Q(t)\rangle,\\
		&	\langle U(t),\mathcal GU(t)\rangle
			-Q(t)^\top\Lambda^{-1}Q(t)
			=\langle U(t)-V_*Q(t),\mathcal GU(t)\rangle+
			\langle V_*Q(t),\mathcal G(U(t)-V_*Q(t))\rangle.
		\end{aligned}
	\end{equation*}
	The Cauchy--Schwarz and Poincar\'e inequalities, the boundedness of
	the trajectory, and the boundedness of $\mathcal G$ yield
	\begin{equation*}
		\begin{aligned}
			\left\|\langle V_*,U(t)\rangle-Q(t)\right\|
			&\leqslant C\|U(t)-V_*Q(t)\|_a\longrightarrow0,\\
			\left\|\langle U(t),\mathcal GU(t)\rangle
			-Q(t)^\top\Lambda^{-1}Q(t)\right\|
			&\leqslant C\|U(t)-V_*Q(t)\|_a\longrightarrow0.
		\end{aligned}
	\end{equation*}
	Consequently,
	\begin{equation*}
		\begin{aligned}
			\langle U(t),\mathcal GU(t)\rangle
			-\left(\frac{1}{\lambda_N}-\frac{\varepsilon}{4}\right)
			\langle U(t),U(t)\rangle
			-
			Q(t)^\top
			\left(
			\Lambda^{-1}
			-\left(\frac{1}{\lambda_N}-\frac{\varepsilon}{4}\right)I_N
			\right)Q(t)
			\longrightarrow0.
		\end{aligned}
	\end{equation*}
	The matrix inside the orthogonal conjugation is bounded below by
	$\varepsilon I_N/4$. Hence, after choosing $T_{\varepsilon}>0$
	sufficiently large,
	\begin{equation}\label{eq: continuous asymptotic spectral separation}
		\left\langle U(t),\mathcal{G}U(t)\right\rangle
		\geqslant
		\left(\frac{1}{\lambda_N}-\frac{\varepsilon}{4}\right)
		\left\langle U(t),U(t)\right\rangle,
		\qquad t\geqslant T_{\varepsilon}.
	\end{equation}
	Since $\langle U(t),U(t)\rangle\geqslant I_N$,
	\eqref{eq:orthogonality time derivative} and cyclic invariance of
	the trace give
	\begin{equation*}
		\begin{aligned}
			\frac{\mathrm{d}}{\mathrm{d}t}
			\left\|\langle U(t),U(t)\rangle-I_N\right\|^2
			&=
			-4\operatorname{tr}\left(
			\langle U(t),\mathcal GU(t)\rangle
			\left(\langle U(t),U(t)\rangle-I_N\right)^2
			\right)\\
			&\leqslant
			-4\left(\frac{1}{\lambda_N}-\frac{\varepsilon}{4}\right)
			\left\|\langle U(t),U(t)\rangle-I_N\right\|^2.
		\end{aligned}
	\end{equation*}
	Gr\"onwall's inequality yields
	\begin{equation*}
		\left\|\langle U(t),U(t)\rangle-I_N\right\|
		\leqslant
		C_{\varepsilon}
		\exp\left(
		-2\left(\frac{1}{\lambda_N}-\frac{\varepsilon}{4}\right)t
		\right),
		\qquad t\geqslant T_{\varepsilon}.
	\end{equation*}

	Let $P_*$ be the $L^2$-orthogonal projection onto
	$\operatorname{span}\{v_1,\ldots,v_N\}$. The eigenfunction
	expansion shows that $P_*$ commutes with $\mathcal G$ and is
	orthogonal with respect to both $(\cdot,\cdot)$ and
	$(\cdot,\cdot)_a$. Moreover,
	\begin{equation*}
		\begin{aligned}
			\left\|(I-P_*)\frac{\mathrm{d}U(t)}{\mathrm{d}t}\right\|^2
			&\leqslant
			\frac{1}{\lambda_{N+1}}
			\left\|(I-P_*)\frac{\mathrm{d}U(t)}{\mathrm{d}t}\right\|_a^2,\\
			\left\|P_*\frac{\mathrm{d}U(t)}{\mathrm{d}t}\right\|^2
			&\leqslant
			\frac{1}{\lambda_1}
			\left\|P_*\frac{\mathrm{d}U(t)}{\mathrm{d}t}\right\|_a^2.
		\end{aligned}
	\end{equation*}
	The convergence above also implies
	\begin{equation*}
		\|(I-P_*)U(t)\|_a
		\leqslant
		\|U(t)-V_*Q(t)\|_a
		\longrightarrow0.
	\end{equation*}
	Since $\langle V_*,U(t)\rangle-Q(t)\to0$,
	$\langle V_*,U(t)\rangle$ is invertible and its inverse is
	uniformly bounded after increasing $T_{\varepsilon}$.
	Using $P_*U(t)=V_*\langle V_*,U(t)\rangle$, the orthogonal
	decomposition and the evolution equation give
	\begin{equation*}
		\begin{aligned}
			\langle V_*,U(t)\rangle^\top
			\left\langle V_*,\frac{\mathrm{d}U(t)}{\mathrm{d}t}\right\rangle
			&\quad=
			-\left(\langle U(t),U(t)\rangle-I_N\right)
			\langle U(t),\mathcal GU(t)\rangle\\
			&\qquad-
			\left\langle
			(I-P_*)U(t),
			(I-P_*)\frac{\mathrm{d}U(t)}{\mathrm{d}t}
			\right\rangle.
		\end{aligned}
	\end{equation*}
	The Cauchy--Schwarz inequality and the uniform matrix bounds
	therefore imply
	\begin{equation*}
		\begin{aligned}
			\left\|P_*\frac{\mathrm{d}U(t)}{\mathrm{d}t}\right\|_a^2
			\leqslant{}
			C\left\|\langle U(t),U(t)\rangle-I_N\right\|^2
			+
			C\|(I-P_*)U(t)\|_a^2
			\left\|\frac{\mathrm{d}U(t)}{\mathrm{d}t}\right\|_a^2.
		\end{aligned}
	\end{equation*}
	Hence
	\begin{equation*}
		\begin{aligned}
			\left\|\frac{\mathrm{d}U(t)}{\mathrm{d}t}\right\|^2
			\leqslant
			\frac{1}{\lambda_{N+1}}
			\left\|\frac{\mathrm{d}U(t)}{\mathrm{d}t}\right\|_a^2
			+
			\left(\frac{1}{\lambda_1}
			-\frac{1}{\lambda_{N+1}}\right)
			\left\|P_*\frac{\mathrm{d}U(t)}{\mathrm{d}t}\right\|_a^2.
		\end{aligned}
	\end{equation*}
	Because $\|(I-P_*)U(t)\|_a\to0$, the last two estimates and the
	orthogonality bound imply, after increasing $T_{\varepsilon}$,
	\begin{equation*}
		\begin{aligned}
			\left\|\frac{\mathrm{d}U(t)}{\mathrm{d}t}\right\|^2
			\leqslant
			\left(\frac{1}{\lambda_{N+1}}
			+\frac{\varepsilon}{4}\right)
			\left\|\frac{\mathrm{d}U(t)}{\mathrm{d}t}\right\|_a^2+
			C_{\varepsilon}
			\exp\left(
			-4\left(\frac{1}{\lambda_N}
			-\frac{\varepsilon}{4}\right)t
			\right).
		\end{aligned}
	\end{equation*}

\end{proof}

\begin{proof}[Proof of Lemma \ref{lem: matrix bounds}]
	For $x\in\mathbb R^N$, the spectral bounds for $\mathcal G$ give
	\begin{equation*}
		\begin{aligned}
			x^\top A_nx
			&=(U_nx,\mathcal GU_nx)
			\leqslant\frac{1}{\lambda_1}\|U_nx\|^2,\\
			x^\top B_nx
			&=\|\mathcal GU_nx\|^2
			\leqslant\frac{1}{\lambda_1}(U_nx,\mathcal GU_nx).
		\end{aligned}
	\end{equation*}
	Hence $0\leqslant A_n\leqslant\lambda_1^{-1}
	\langle U_n,U_n\rangle$ and
	$0\leqslant B_n\leqslant\lambda_1^{-1}A_n$. If
	$U_n\in\mathcal M_{\geqslant}^N$, then $U_n$ has full column rank,
	so $A_n$ and $B_n$ are positive definite. The remaining bounds follow
	from $0\leqslant O_n\leqslant\langle U_n,U_n\rangle$.
\end{proof}

\begin{proof}[Proof of Lemma \ref{lem:discrete local estimates}]
	The contraction estimate \eqref{eq: orthogonality error relation} and $\tau_n\geqslant\tau_{\min}$ yield
	\begin{equation}\label{eq: new On bound}
		\|O_n\|
		\leqslant
		C\left(1+\frac{\tau_{\min}}{E(U_0)}\right)^{-n}.
	\end{equation}

		For all sufficiently large $n$, define
		\begin{equation*}
			\bar U_n\coloneqq U_n(I_N+O_n)^{-1/2}.
		\end{equation*}
		Since $I_N+O_n=\langle U_n,U_n\rangle\geqslant I_N$,
		\begin{equation*}
			\langle\bar U_n,\bar U_n\rangle=I_N.
		\end{equation*}
		 The boundedness of $\{U_n\}$ and $O_n\to0$ therefore give
		\begin{equation*}
			\|\bar U_n-U_n\|_a
			\leqslant
			\|U_n\|_a
			\left\|(I_N+O_n)^{-1/2}-I_N\right\|
			\leqslant C\|O_n\|.
		\end{equation*}
		Hence $[\bar U_n]\to[V_*]$ by Theorem \ref{thm: convergence of Un}.

		Let $P_*$ be the $L^2$-orthogonal projection onto $\operatorname{span}\{v_1,\ldots,v_N\}$. Since $P_*$ commutes with $\mathcal G$,
		\begin{equation*}
			\begin{aligned}
				(I-P_*)\nabla_GE_{\mathcal G}(\bar U_n)
			={}\mathcal G(I-P_*)\bar U_n-(I-P_*)\bar U_n
			\langle\bar U_n,\mathcal G\bar U_n\rangle
		\end{aligned}
	\end{equation*}
	and
	\begin{equation*}
		\lambda_{\min}\left(
		\langle\bar U_n,\mathcal G\bar U_n\rangle
			\right)
			\longrightarrow\frac{1}{\lambda_N}.
		\end{equation*}
		Indeed, after aligning $\bar U_n$ with $V_*$ by an orthogonal matrix, the convergence in $\big(H_0^1(\Omega)\big)^N$ and $\mathcal GV_*=V_*\Lambda^{-1}$ imply convergence of the eigenvalues of $\langle\bar U_n,\mathcal G\bar U_n\rangle$ to those of $\Lambda^{-1}$.

		Choose $R_n\in\mathcal{O}^N$ such that
		\begin{equation*}
			R_n^\top\langle\bar U_n,\mathcal G\bar U_n\rangle R_n
			=\operatorname{diag}(\lambda_{1,n},\ldots,\lambda_{N,n}),
		\end{equation*}
		and expand
		\begin{equation*}
			(I-P_*)\bar U_nR_n
			=
			\sum_{j=N+1}^{\infty}
			v_j\big(c_{j1}^{(n)},\ldots,c_{jN}^{(n)}\big).
		\end{equation*}
		For all sufficiently large $n$,
		\begin{equation*}
			\lambda_{i,n}
			\geqslant
			\frac12\left(
			\frac{1}{\lambda_N}+\frac{1}{\lambda_{N+1}}
			\right),
			\qquad i=1,\ldots,N.
		\end{equation*}
		Consequently,
		\begin{equation*}
			\begin{aligned}
				\left\|(I-P_*)\nabla_GE_{\mathcal G}(\bar U_n)\right\|_a^2
				&=
				\sum_{i=1}^N\sum_{j=N+1}^{\infty}
				\lambda_j
				\left(\frac{1}{\lambda_j}-\lambda_{i,n}\right)^2
				\left|c_{ji}^{(n)}\right|^2\\
				&\geqslant
				\frac14\left(
				\frac{1}{\lambda_N}-\frac{1}{\lambda_{N+1}}
				\right)^2
				\left\|(I-P_*)\bar U_n\right\|_a^2.
			\end{aligned}
		\end{equation*}
		This proves
		\begin{equation*}
			\left\|(I-P_*)\nabla_GE_{\mathcal G}(\bar U_n)\right\|_a
		\geqslant
		\frac12\left(
		\frac{1}{\lambda_N}-\frac{1}{\lambda_{N+1}}
			\right)
			\left\|(I-P_*)\bar U_n\right\|_a.
		\end{equation*}

		In this proof, set $C_n=\langle V_*,\bar U_n\rangle$. Since
		\begin{equation*}
			P_*\bar U_n=V_*C_n,
			\qquad
			C_n^\top C_n
			+\langle(I-P_*)\bar U_n,(I-P_*)\bar U_n\rangle
			=I_N,
		\end{equation*}
		both $I_N-C_n^\top C_n$ and $I_N-C_nC_n^\top$ are positive semidefinite, and
		\begin{equation*}
			\operatorname{tr}(I_N-C_nC_n^\top)
			=\|(I-P_*)\bar U_n\|^2.
		\end{equation*}
		The eigenfunction expansion gives the identity
		\begin{equation*}
			\begin{aligned}
				2\big(E(\bar U_n)-E(V_*)\big)
				={}
				\left\|(I-P_*)\bar U_n\right\|_a^2
				-\operatorname{tr}\left(
				\Lambda(I_N-C_nC_n^\top)
				\right).
			\end{aligned}
		\end{equation*}
		Moreover,
		\begin{equation*}
			0
			\leqslant
			\operatorname{tr}\left(
			\Lambda(I_N-C_nC_n^\top)
			\right)
			\leqslant
			\frac{\lambda_N}{\lambda_{N+1}}
			\left\|(I-P_*)\bar U_n\right\|_a^2.
		\end{equation*}
		Therefore,
		\begin{equation*}
			\left(1-\frac{\lambda_N}{\lambda_{N+1}}\right)
			\left\|(I-P_*)\bar U_n\right\|_a^2
			\leqslant
			2\big(E(\bar U_n)-E(V_*)\big)
			\leqslant
			\left\|(I-P_*)\bar U_n\right\|_a^2.
		\end{equation*}

		For all sufficiently large $n$, $C_n$ is invertible. Let
		\begin{equation*}
			Q_n=C_n(C_n^\top C_n)^{-1/2}\in\mathcal{O}^N.
		\end{equation*}
		If $\sigma_i(C_n)$ are the singular values of $C_n$, then
		\begin{equation*}
			\begin{aligned}
				\|C_n-Q_n\|^2
				&=\sum_{i=1}^N(1-\sigma_i(C_n))^2\\
				&\leqslant
				\sum_{i=1}^N(1-\sigma_i(C_n)^2)
				=\|(I-P_*)\bar U_n\|^2.
			\end{aligned}
		\end{equation*}
		The $a$-orthogonality of the ranges of $P_*$ and $I-P_*$ now gives
		\begin{equation*}
			\begin{aligned}
				\|[\bar U_n]-[V_*]\|_a^2
				&\leqslant
				\|\bar U_n-V_*Q_n\|_a^2\\
				&=
				\|(I-P_*)\bar U_n\|_a^2
				+\|V_*(C_n-Q_n)\|_a^2\\
				&\leqslant
				\left(1+\frac{\lambda_N}{\lambda_{N+1}}\right)
				\|(I-P_*)\bar U_n\|_a^2.
			\end{aligned}
		\end{equation*}
		Combining the preceding estimates yields
		\begin{equation*}
			\begin{aligned}
				E(\bar U_n)-E(V_*)
			&\leqslant
			C\|\nabla_GE_{\mathcal G}(\bar U_n)\|_a^2,\\
			\|[\bar U_n]-[V_*]\|_a^2
			&\leqslant
			C\big(E(\bar U_n)-E(V_*)\big).
			\end{aligned}
		\end{equation*}

		A direct expansion of the definition of $\mathcal A_UU$ shows that, on every bounded subset of $\big(H_0^1(\Omega)\big)^N$,
		\begin{equation*}
			\begin{aligned}
				\mathcal A_UU-\mathcal A_VV
				={}&\mathcal G(U-V)\langle U,U\rangle
				+\mathcal GV\big(\langle U,U\rangle-\langle V,V\rangle\big)\\
				&-(U-V)\langle U,\mathcal GU\rangle
				-V\big(\langle U,\mathcal GU\rangle
				-\langle V,\mathcal GV\rangle\big).
			\end{aligned}
		\end{equation*}
		The boundedness of $\mathcal G$ and the Cauchy--Schwarz inequality therefore give
		\begin{equation*}
			\|\mathcal A_UU-\mathcal A_VV\|_a
			\leqslant C\|U-V\|_a.
		\end{equation*}
		Since $\bar U_n\in\mathcal M^N$,
		\begin{equation*}
			\begin{aligned}
				\|\nabla_GE_{\mathcal G}(\bar U_n)\|_a
				=\|\mathcal A_{\bar U_n}\bar U_n\|_a\
				\leqslant
				\|\mathcal A_{U_n}U_n\|_a+C\|O_n\|.
			\end{aligned}
		\end{equation*}
		Moreover, the identity $U_n=\bar U_n(I_N+O_n)^{1/2}$ gives
	\begin{equation*}
		E(U_n)-E(\bar U_n)
		=\frac12\operatorname{tr}\left(
		\langle\bar U_n,\bar U_n\rangle_aO_n
			\right).
		\end{equation*}
		The matrices $O_n$ and $\langle\bar U_n,\bar U_n\rangle_a$ are positive semidefinite, and
		\begin{equation*}
			\lambda_1I_N
			\leqslant
			\langle\bar U_n,\bar U_n\rangle_a
			\leqslant CI_N.
		\end{equation*}
		Hence
		\begin{equation*}
			\begin{aligned}
				E(U_n)-E(V_*)
				&\geqslant
				\frac{\lambda_1}{2}\operatorname{tr}(O_n)
				\geqslant\frac{\lambda_1}{2}\|O_n\|,\\
				E(U_n)-E(V_*)
				&\leqslant
				E(\bar U_n)-E(V_*)+C\|O_n\|.
			\end{aligned}
			\end{equation*}
			The local gradient bound above and the energy error bound for $\bar U_n$ now imply
			\begin{equation*}
				E(\bar U_n)-E(V_*)
				\leqslant
				C\|\mathcal A_{U_n}U_n\|_a^2+C\|O_n\|^2.
			\end{equation*}
			Since $\|O_n\|\to0$, the quadratic term is bounded by $C\|O_n\|$ for all sufficiently large $n$. Hence
			\begin{equation*}
				E(U_n)-E(V_*)
				\leqslant
			C\|\mathcal A_{U_n}U_n\|_a^2+C\|O_n\|.
		\end{equation*}
		Furthermore, $\|O_n\|\to0$ and the lower energy bound above imply
		\begin{equation*}
			\|U_n-\bar U_n\|_a^2
			\leqslant C\|O_n\|^2
			\leqslant C\big(E(U_n)-E(V_*)\big).
		\end{equation*}
		Therefore,
		\begin{equation*}
			\begin{aligned}
				\|[U_n]-[V_*]\|_a^2
				&\leqslant
				2\|U_n-\bar U_n\|_a^2
				+2\|[\bar U_n]-[V_*]\|_a^2\\
				&\leqslant
				C\big(E(U_n)-E(V_*)\big).
			\end{aligned}
		\end{equation*}
		Finally, for every $Q\in\mathcal O^N$,
		\begin{equation*}
			\mathcal A_{UQ}(UQ)=(\mathcal A_UU)Q,
			\qquad
			\mathcal A_{V_*Q}(V_*Q)=0.
		\end{equation*}
		Choose an orthogonal matrix that realizes the Grassmann distance. The local Lipschitz estimate and the preceding distance bound give
		\begin{equation*}
			\|\mathcal A_{U_n}U_n\|_a^2
			\leqslant
			C\|[U_n]-[V_*]\|_a^2
			\leqslant
			C\big(E(U_n)-E(V_*)\big).
		\end{equation*}
\end{proof}

		% References
	\bibliographystyle{siamplain}
	\bibliography{references}
	
\end{document}